\documentclass[12pt]{amsart}
\usepackage{amssymb}
\usepackage{amsfonts}
\usepackage{latexsym}
\usepackage{amscd}

\usepackage{tikz}

\usepackage{color,graphics}

\usepackage[mathscr]{euscript}
\usepackage{xy} \xyoption{all}

\newcommand{\N}{{\mathbb{N}}}
\newcommand{\Z}{{\mathbb{Z}}}

\newcommand{\C}{{\mathbb{C}}}

\newcommand{\calV}{{\mathcal V}}

\newcommand{\calO}{{\mathcal O}}

\newcommand{\gotA}{{\mathfrak A}}
\newcommand{\gotAA}{{\mathfrak A _0}}
\newcommand{\gotAAp}{{\mathfrak A '_0}}
\newcommand{\gotAAA}{{\mathfrak A _{00}}}
\newcommand{\gotB}{{\mathfrak B}}

\newcommand{\ol}{\overline}
\newcommand{\uloopr}[1]{\ar@'{@+{[0,0]+(-4,5)}@+{[0,0]+(0,10)}@+{[0,0] +(4,5)}}^{#1}}
\newcommand{\uloopd}[1]{\ar@'{@+{[0,0]+(5,4)}@+{[0,0]+(10,0)}@+{[0,0]+ (5,-4)}}^{#1}}
\newcommand{\dloopr}[1]{\ar@'{@+{[0,0]+(-4,-5)}@+{[0,0]+(0,-10)}@+{[0, 0]+(4,-5)}}_{#1}}
\newcommand{\dloopd}[1]{\ar@'{@+{[0,0]+(-5,4)}@+{[0,0]+(-10,0)}@+{[0,0 ]+(-5,-4)}}_{#1}}

\newcommand{\luloop}[1]{\ar@'{@+{[0,0]+(-8,2)}@+{[0,0]+(-10,10)}@+{[0, 0]+(2,2)}}^{#1}}

\newcommand{\dotedge}{\ar@{.}}
\newcommand{\eqedge}{\ar@{=}}

\newcommand{\mon}[1]{\calV(#1)} 

\newcommand{\Si}{{\rm Sink}}

\DeclareMathOperator{\rd}{red}
\newcommand{\Cstred}{C^*_{\rd}}
\newcommand{\bgast}{\mbox{\Large$*$}}

\theoremstyle{plain}
\newtheorem{theorem}{Theorem}[section]
\newtheorem{lemma}[theorem]{Lemma}
\newtheorem{proposition}[theorem]{Proposition}
\newtheorem{corollary}[theorem]{Corollary}
\newtheorem{conjecture}[theorem]{Conjecture}

\theoremstyle{definition}
\newtheorem{definition}[theorem]{Definition}
\newtheorem{example}[theorem]{Example}

\newtheorem{remark}[theorem]{Remark}
\newtheorem*{remark*}{Remark}
\newtheorem{remarks}[theorem]{Remarks}

\newtheorem*{assumption*}{Assumption}

\numberwithin{equation}{section}

\begin{document}

\title[Purely infinite simple class]{Purely infinite simple reduced C*-algebras of one-relator separated graphs}%

\author{Pere Ara}

\address{Departament de Matem\`atiques, Universitat Aut\`onoma de Barcelona,
08193 Bellaterra (Barcelona), Spain.} \email{para@mat.uab.cat}


\thanks{Partially supported by DGI MICIIN-FEDER
MTM2011-28992-C02-01, and by the Comissionat per Universitats i
Recerca de la Generalitat de Catalunya.}



\date{\today}

\begin{abstract}
Given a separated graph $(E,C)$, there are two different C*-algebras
associated to it, the full graph C*-algebra $C^*(E,C)$, and the
reduced one $\Cstred(E,C)$. For a large class of separated graphs
$(E,C)$, we prove that $\Cstred (E,C)$  either is purely infinite
simple or admits a faithful tracial state. The main tool we use to
show pure infiniteness of reduced graph C*-algebras is a
generalization to the amalgamated case of a result on purely
infinite simple free products due to Dykema.
\end{abstract}
\maketitle


\section{Introduction}

Separated graphs have been introduced in \cite{AG} and \cite{AG2}.
Their associated graph algebras provide generalizations of the usual
graph C*-algebras (\cite{Raeburn}) and Leavitt path algebras
(\cite{AA}, \cite{AMP}) associated to directed graphs, although
these algebras behave quite differently from the usual graph
algebras because the range projections associated to different edges
need not commute. One motivation for their introduction was to
provide graph-algebraic models for C*-algebraic analogues of the
Leavitt algebras of \cite{lea}. These had been studied by various
authors over the years, notably McClanahan (see \cite{McCla1},
\cite{McCla2}, \cite{McCla3}). As discussed below, another
motivation was to obtain graph algebras whose structure of
projections is as general as possible.

\medskip

Given a finitely separated graph $(E,C)$, two different graph
C*-algebras are considered in \cite{AG2}, the {\it full} graph
C*-algebra $C^*(E,C)$ and the {\it reduced} graph C*-algebra
$\Cstred (E,C)$. These two C*-algebras agree in the classical,
non-separated case, by \cite[Theorem 3.8(2)]{AG2}, but they differ
generally, see \cite[Section 4]{AG2}. In this paper we obtain
significant progress on some of the open problems raised in
\cite{AG2}. In particular we completely solve \cite[Problem
7.3]{AG2} for the class of one-relator separated graphs, giving a
characterization of the purely infinite simple reduced graph
C*-algebras corresponding to this class of separated graphs. As a
special case, we deduce that the reduced C*-algebras $\Cstred
(E(m,n), C(m,n))$, which are certain C*-completions of the Leavitt
algebras $L_{\C}(m,n)$ of type $(m,n)$, are purely infinite simple
for $m\ne n$ (Corollary \ref{cor:Umn}).

\medskip

Reduced graph C*-algebras of separated graphs are defined as reduced
amalgamated free products of certain usual graph C*-algebras, with
respect to canonical conditional expectations, see Section
\ref{sect:preliminaries}. The reader is referred to \cite{voicu} and
\cite[4.7,4.8]{BO} for a quick introduction to the subject of
reduced amalgamated free products of C*-algebras. Sufficient
conditions for a reduced free product of two C*-algebras to be
purely infinite simple were obtained in \cite{DykemaRordamCJM},
\cite{ChoDyk} and \cite{DykemaScan}. The main tool we use in the
present paper to show our results on reduced graph C*-algebras of
separated graphs is a generalization to certain amalgamated free
products of a result of Dykema \cite{DykemaScan} (see Theorem
\ref{thm:Dykema's}).

\medskip

Let us recall the definition of a separated graph:

\begin{definition} \label{defsepgraph} \cite[Definition 1.3]{AG2}
A \emph{separated graph} is a pair $(E,C)$ where $E$ is a graph,
$C=\bigsqcup _{v\in E^ 0} C_v$, and $C_v$ is a partition of
$s^{-1}(v)$ (into pairwise disjoint nonempty subsets) for every
vertex $v$. (In case $v$ is a sink, we take $C_v$ to be the empty
family of subsets of $s^{-1}(v)$.)

If all the sets in $C$ are finite, we say that $(E,C)$ is a
\emph{finitely separated} graph. This necessarily holds if $E$ is
row-finite.

The set $C$ is a \emph{trivial separation} of $E$ in case $C_v=
\{s^{-1}(v)\}$ for each $v\in E^0\setminus \Si(E)$. In that case,
$(E,C)$ is called a \emph{trivially separated graph} or a
\emph{non-separated graph}.
\end{definition}

By its definition (see Section \ref{sect:preliminaries}), the
projections of $C^*(E,C)$ and $\Cstred (E,C)$ satisfy some obvious
relations, prescribed by the structure of the separated graph
$(E,C)$. These relations can be chosen arbitrarily, and this was one
of the main motivations for the work in \cite{AG} and \cite{AG2}.
This can be formalized as follows. Let $(E,C)$ be a finitely
separated graph, and let $M(E,C)$ be the abelian monoid given by
generators $a_v$, $v\in E^0$, and relations $a_v=\sum _{e\in X}
a_{r(e)}$, for $X\in C_v$, $v\in E^0$. Then there is a canonical
monoid homomorphism $M(E,C)\to \mon{C^*(E,C)}$, which is conjectured
to be an isomorphism for all finitely separated graphs $(E,C)$ (see
Section \ref{sect:k-theory} for a discussion on this problem). Here
$\mon{\mathcal A}$ denotes the monoid of Murray-von Neumann
equivalence classes of projections in $M_{\infty}(\mathcal A)$, for
any C*-algebra $\mathcal A$.

\medskip

Given a presentation $\langle \mathcal X \mid \mathcal R\rangle $ of
an abelian conical monoid $M$, satisfying some natural conditions,
it was shown in \cite[Proposition 4.4]{AG} how to associate to it a
separated graph $(E,C)$ such that $M(E,C)\cong M$. We will now
recall this construction for one-relator monoids. Let
$$\langle a_1,\dots ,a_n \mid \sum _{i=1}^n r_ia_i = \sum_{i=1}^n
s_ia_i \rangle $$ be a presentation of the one-relator abelian
conical monoid $M$, where $a_1,\dots ,a_n$ are free generators, and
$r_i, s_i$ are non-negative integers such that $r_i+s_i>0$ for all
$i$. Let $(E,C)$ be the finitely separated graph constructed as
follows:
\begin{enumerate}
\item $E^0 := \{v,w_1,w_2,\dots ,w_n\}$.
\item  $v$ is a source, and all the $w_i$ are sinks.
\item For each $i\in \{1,\dots ,n\}$, there are
exactly $r_i+s_i$ edges with source $v$ and range $w_i$.
\item $C=C_v=\{X,Y\}$, where $X$ contains exactly $s_i$ edges $v \rightarrow w_i$ for each $i$, and $Y$
contains exactly $r_i$ edges $v \rightarrow v_i$ for each $i$. Thus,
$E^1= X\sqcup Y$.
\end{enumerate}

\medskip

We call a separated graph constructed in this way a {\it one-relator
separated graph}. As a particular example, we may consider the
presentation $\langle a \mid ma=na \rangle $, with $1\le m\le n$.
This gives rise to the separated graph $(E(m,n), C(m,n))$ considered
in \cite[Example 4.5]{AG2}, with two vertices $v$ and $w$, and $n+m$
arrows from $v$ to $w$, and with $C(m,n)=\{X,Y\}$ where $|X|=n$ and
$|Y|=m$. The C*-algebras $C^*(E(m,n), C(m,n))$ and $\Cstred (E(m,n),
C(m,n))$ are closely related to the C*-algebras studied in
\cite{Brown, McCla1, McCla2, McCla3}, see \cite[Sections 4 and
6]{AG2}.

\medskip

In this paper, we show the following dichotomy for the reduced graph
C*-algebras of one-relator separated graphs:

\begin{theorem}
\label{thm:main} Let $(E,C)$ be the one-relator separated graph
associated to the presentation $$\langle a_1,\dots ,a_n \mid \sum
_{i=1}^n r_ia_i=\sum _{i=1}^n s_ia_i \rangle .$$ Set $M=\sum
_{i=1}^n r_i$ and $N=\sum _{i=1}^n s_i$, and assume that $2\le M\le
N$. Then $\Cstred (E,C)$ either is purely infinite simple or has a
faithful tracial state, and it is purely infinite simple if and only
if $M<N$ and there is $i_0\in \{1,\dots ,n\}$ such that $s_{i_0}>0$
and $r_{i_0}>0$. Moreover, if $N+M\ge 5$ and $\Cstred (E,C)$ is
finite, then it is simple with a unique tracial state.
\end{theorem}

When $N+M\le 4$, the C*-algebra $\Cstred (E,C)$ is also simple,
except for a few cases. We analyze the different possibilities in
the final part of Section \ref{sect:finite-case}. The case where
$M=1$ corresponds to an ordinary graph C*-algebra.  Observe that in
particular we get that the algebras $\Cstred (E(m,n), C(m,n))$ are
purely infinite simple when $1<m<n$. It was suggested in
\cite[Example 4.3]{McCla3} that the simple C*-algebras
$U^{\text{nc}}_{(m,n), \text{red}}$, for $1<m<n$, might be examples
of finite but not stably finite C*-algebras. In view of
\cite[Proposition 6.1]{AG2}, our main result shows in particular
that the C*-algebras $U^{\text{nc}}_{(m,n), \text{red}}$ are purely
infinite if $m<n$.

\medskip

It is worth to mention the appearance of some group C*-algebras as
graph C*-algebras of separated graphs. As noted in \cite{AG2}, one
example of this situation occurs when we consider the separated
graph $(E,C)$ with just one vertex and with the sets of the
partition reduced to singletons. In this case, the full graph
C*-algebra $C^*(E,C)$ is just the full group C*-algebra $C^*(\mathbb
F )$ of a free group $\mathbb F$ of rank $|E^1|$, while the reduced
graph C*-algebra is precisely the reduced group C*-algebra $C^*_r
(\mathbb F)$. In the present investigation, we discover another
situation in which the full and the reduced graph C*-algebras
correspond (through a Morita-equivalence) to the full and reduced
group C*-algebras of a group, respectively (see Lemma
\ref{lem:full-algebras}(2) and Proposition
\ref{prop:reducedgpCstar}). The universal unital C*-algebra
generated by a partial isometry also appears as a full corner of the
full C*-algebra of a one-relator separated graph (see Lemma
\ref{lem:full-algebras}(1)).

\medskip

We now outline the contents of the paper. After a section of
preliminaries, we obtain in Section \ref{sect:dykema} the
generalization to certain amalgamated free products of Dykema's
result (\cite[Theorem 3.1]{DykemaScan}). Section
\ref{sect:pur-infinite} contains the proof of pure infiniteness and
simplicity for a class of reduced graph C*-algebras of one-relator
separated graphs (Theorem \ref{thm:motivating}). It also contains a
direct application of Dykema's Theorem to show that certain reduced
free products of Cuntz algebras are purely infinite simple
(Proposition \ref{prop:one-vertex}). We study the finite cases in
Section \ref{sect:finite-case}, obtaining in particular the proof of
our main result (Theorem \ref{thm:main}). Finally, in Section
\ref{sect:k-theory}, we briefly discuss an open problem raised in
\cite{AG2} on the non-stable K-theory of $C^*(E,C)$, and we point
out some K-theoretic consequences of our results. In particular, we
give an example of a one-relator separated graph $(E,C)$ such that
the full graph C*-algebra $C^*(E,C)$ is stably finite (indeed
residually finite dimensional), while the reduced graph C*-algebra
$\Cstred (E,C)$ is purely infinite simple (Example
\ref{examplM(E,C)st}).

\section{Preliminaries}
\label{sect:preliminaries}

Throughout, all graphs will be directed graphs of the form $E=
(E^0,E^1,s,r)$, where $E^0$ and $E^1$ denote the sets of vertices
and edges of $E$, respectively, and $s,r: E^1 \rightarrow E^0$ are
the source and range maps. We follow the convention of composing
paths from left to right -- thus, a path in $E$ is given in the form
$\alpha= e_1e_2 \cdots e_n$ where the $e_i\in E^1$ and $r(e_i)=
s(e_{i+1})$ for $i<n$. The \emph{length} of such a path is $|\alpha|
:= n$. Paths of length $0$ are identified with the vertices of $E$.

\begin{definition} \label{def:LPASG} \cite[Definition 1.5]{AG2}
For any separated graph $(E,C)$, the full graph C*-algebra of the
separated graph $(E,C)$ is the universal C*-algebra with generators
$\{ v, e\mid v\in E^0,\ e\in E^1 \}$, subject to the following
relations:
\begin{enumerate}
\item[] (V)\ \ $vw = \delta_{v,w}v$ \ and $v=v^*$ \ for all $v,w \in E^0$ ,
\item[] (E)\ \ $s(e)e=er(e)=e$ \ for all $e\in E^1$ ,
\item[] (SCK1)\ \ $e^*f=\delta _{e,f}r(e)$ \ for all $e,f\in X$, $X\in C$, and
\item[] (SCK2)\ \ $v=\sum _{ e\in X }ee^*$ \ for every finite set $X\in C_v$, $v\in E^0$.
\end{enumerate}
\end{definition}

In case $(E,C)$ is trivially separated, $C^*(E,C)$ is just the
classical graph C*-algebra $C^*(E)$.

We now recall the important abstract characterization of the reduced
amalgamated free product of C*-algebras, due to Voiculescu
\cite{voicu}.

\begin{definition}  \label{def:redamprod}
The \emph{reduced amalgamated product} $(A, \Phi )$ of a nonempty
family $(A_{\iota}, \Phi _{\iota})_{\iota \in I}$ of unital
C*-algebras containing a unital subalgebra $A_0$ with conditional
expectations $\Phi _{\iota}\colon A_{\iota}\to A_0$ is uniquely
determined by the following conditions:
\begin{enumerate}
\item $A$ is a unital C*-algebra, and there are unital $*$-homomorphisms
$\sigma _{\iota}\colon A_{\iota} \to A$ such that $\sigma_{\iota}|
_{A_0}= \sigma_{\iota'}|_{A_0}$ for all $\iota,\iota '\in I$.
Moreover the map $\sigma_{\iota}| _{A_0}$ is injective and we
identify $A_0$ with its image in $A$ through this map.
\item $A$ is generated by $\bigcup_{\iota \in I}
\sigma_{\iota}(A_{\iota})$.
\item $\Phi\colon A\to A_0$ is a
conditional expectation such that $\Phi \circ \sigma_{\iota}= \Phi
_{\iota}$ for all $\iota \in I$.
\item For $(\iota _1,\dots ,\iota _n)\in \Lambda (I)$ and $a_j\in \ker{\Phi_{\iota _j}}$
we have $\Phi (\sigma_{\iota _1}(a_1)\cdots \sigma _{\iota
_n}(a_n))=0$. Here, $\Lambda (I)$ denotes the set of all finite
tuples $(\iota _1,\dots ,\iota _n) \in \bigsqcup_{n=1}^\infty I^n$
such that $\iota _i\ne \iota _{i+1}$ for $i=1,\dots ,n-1$.
\item If $c\in A$ is such that $\Phi (a^*c^*ca)=0$ for all $a\in A$,
then $c=0$.
\end{enumerate}
\end{definition}

As usual, we will write $A^{{\rm o}}=\ker (\Phi)$, where $\Phi $ is
a conditional expectation from $A$ onto a unital subalgebra $C$.
Given subsets $T_1$, $T_2$ of an algebra $\mathcal H$, we will
denote by $\Lambda ^{{\rm o}}(T_1,T_2)$ the set of all elements of
$\mathcal H$ of the form $a_1a_2 \cdots a_r$, where $a_j\in T_{i_j}$
and $i_1\ne i_2$, $i_2\ne i_3$,\dots , $i_{r-1}\ne i_r$.

\medskip

We can now recall the definition of the reduced graph C*-algebra
$\Cstred (E,C)$ of a finitely separated graph. We only need here the
case in which $E^0$ is a finite set. In this case, all the
C*-algebras involved in the definition are unital, and this
simplifies somewhat the definition. The reader is referred to
\cite{AG2} for the general case.

\medskip

Let $(E,C)$ be a finitely separated graph with $|E^0|<\infty $.  Set
$A_0=C(E^0)$ and, for $X\in C$,  $A_X=C^*(E_X)$, where $E_X$ is the
subgraph of $E$ with $(E_X)^0=E^0$ and $(E_X)^1= X$.

\begin{definition}
\label{def:redCstar} Let $(E,C)$ be a finitely separated graph, with
$|E^0|<\infty$,  and let $A_0, A_X$ be as defined above, for $X\in
C$. Let $\Phi_X\colon C^*(E_X)\to C(E^0)$ be the canonical
conditional expectations defined in \cite[Theorem 2.1]{AG2}. Then
the {\it reduced graph C*-algebra} $(\Cstred (E,C),\Phi )$
associated to $(E,C)$ is the reduced amalgamated product of the
family $(C^*(E_X),\Phi_X)_{X\in C}$. Since all the conditional
expectations $\Phi _X$ are faithful, it follows from \cite[Theorem
2.1]{ivanov} that the canonical conditional expectation $\Phi \colon
\Cstred (E,C)\to C(E^0)$ is also faithful.
\end{definition}

A concrete description in terms of a reduced amalgamated product of
finite-dimensional C*-algebras will be given in Section
\ref{sect:pur-infinite} for the reduced graph C*-algebras associated
to one-relator separated graphs.

We will make use of the following result and notation:

\begin{theorem}
\label{thm:freeproducps} \cite[Theorem 4.8.5]{BO} Let $1\in
D\subseteq A_i$ and nondegenerate conditional expectations $E_i^{A}$
from $A_i$ onto $D$ be given, for $i=1,2$. Assume $1\in D\subseteq
B_i$ and assume there exist nondegenerate conditional expectations
$E_i^B$ from $B_i$ onto $D$. Let $\theta _i\colon A_i\to B_i$ be
u.c.p maps such that $(\theta)_i|_D=\text{id}_D$ and $E_i^B\circ
\theta _i= E_i^A$. Then, there is a u.c.p. map
$$\theta_1 \ast \theta _2\colon (A_1,E_1^A)*_D(A_2,
E_2^A)\longrightarrow (B_1, E_1^B)*_D (B_2,E_2^B)$$ such that
$(\theta_1 \ast \theta _2)|_D =\text{id}_D$ and
$$(\theta_1\ast \theta _2)(a_1a_2\cdots a_n)= \theta
_{i_1}(a_1)\theta_{i_2}(a_2)\cdots \theta _{i_n}(a_n)$$ for $a_j\in
A^{{\rm o}}_{i_j}$ with $i_1\ne i_2\ne \cdots \ne i_n$.
\end{theorem}

\section{An adaptation of Dykema's Theorem.}
\label{sect:dykema}

In this section, we adapt Dykema's result in \cite{DykemaScan} to
the amalgamated case, and indeed we generalize the range of
applications, since our hypothesis are a bit different.

Let $(A, \Phi _A)$ and $(B, \Phi _B)$ be C*-algebras with faithful
conditional expectations $\Phi_A\colon A\to C$ and $\Phi _B\colon
B\to C$, where $C$ is a common C*-subalgebra of $A$ and $B$.
Consider the C*-algebra reduced amalgamated free product
$$(\gotA, \Phi)=(A, \Phi _A)*(B, \Phi _B).$$
Let $P\in C$ be a central projection in $C$ such that $PC=P\C$, and
let $\gamma \colon C\to \C$ be a faithful state on $C$. Then $\phi
:= \gamma \circ \Phi$ is a faithful state on $\gotA$, since by
\cite[Theorem 2.1]{ivanov}, $\Phi$ is also faithful.

Assume that there exists a partial isometry $v$ in $A$ such that
$v^*v=p$ and $vv^*=q$, where $p$ and $q$ are projections in $A$ such
that $p\le P$ and $q\le 1-p$ and there is $0<\lambda $ such that
$\phi (vx)=\lambda \phi (xv)$ for all $x\in A$.

Set $A_0:=pAp+(1-p)A(1-p)$ and $\gotAA:= C^*(A_0\cup B)$. We assume
that there is $y\in \gotAA$ such that  $yy^*=q$, $p_1:= y^*y\le p$,
so that $p_1\in \gotAA$. Assume moreover there is $0<\mu $ such that
$\phi (y^*x)=\mu \phi (xy^*)$ for all $x\in \gotA$, and that
$\lambda \mu <1$.

Define $w:= y^*v$, so that $ww^*=p_1$ and $w^*w= p$. Define
$p_n:=w^n(w^*)^n$ and note that $p_n\le p_{n-1}$ and $p_n\in p\gotA
p$ for all $n\ge 1$ (where $p_0:=p$).

Let $E\colon \gotA\to A_0$ be the composition of the conditional
expectations $E^A_{\gotA}\colon \gotA\to A$ and $E_A^{A_0}\colon
A\to A_0$, where $E^A_{\gotA}=\text{id}_A*\Phi _B$ is the canonical
conditional expectation given by Theorem \ref{thm:freeproducps}, and
$$E_A^{A_0}(a)=pap+(1-p)a(1-p)$$
for $a\in A$.

\begin{lemma}
\label{lem:compatiblePhi} With the above notation, we have
$\Phi\circ E=\Phi$, and in particular $\phi \circ E=\phi$.
\end{lemma}

\begin{proof}
It suffices to show that $\Phi_A (pa(1-p))=0$ for all $a\in A$. For
$a\in A$ we have
\begin{align*}
\Phi_A (pa(1-p)) &  = \Phi_A(Ppa(1-p))=\phi (pa(1-p))\phi (P)^{-1}P
\\ & = \phi (v^*va(1-p))\phi (P)^{-1}P
\\ & =\phi(va(1-p)v^*)(\lambda \phi (P))^{-1}P=0 ,
\end{align*}
showing the result.
\end{proof}

The hypothesis of the theorem of this section involves the following
concept.

\begin{definition}
\label{def:A-free} In the above situation, let $z$ be a projection
in $p\gotA p$. Then $z$ is said to be {\it $A$-free} if
\begin{enumerate}
\item $E(z)\in \C \cdot p$.
\item For any word $a\in \Lambda^{{\rm o}}((pAp) ^{{\rm o}}, C^*(z,p)^{{\rm o}})$ of length $>1$, we have
$E(a)=0$.
\end{enumerate}
\end{definition}

We can state now the following result, which is a generalization of
\cite[Theorem 3.1]{DykemaScan} to the amalgamated case.

\begin{theorem}
\label{thm:Dykema's} Let $(A, \Phi _A)$ and $(B,\Phi_B)$ be
C*-algebras with conditional expectations satisfying the conditions
stated above, and let $p$, $q$, $p_k=w^k(w^*)^k$, $k\ge 1$, be the
projections defined above. Assume that, for each $k\ge 1$, the
projections $p_k$ are $A$-free. Suppose in addition that $q\gotAA q$
contains a unital diffuse abelian C*-subalgebra which is contained
in the centralizer of $\phi$ in $\gotA$, and that $p$ is full in
$\gotA$. Then $\gotA$ is purely infinite and simple.
\end{theorem}

Before we prove Theorem \ref{thm:Dykema's}, let us show that it
provides a generalization of Dykema's result.

\begin{corollary} \cite[Theorem 3.1]{DykemaScan}
\label{corol:Dykema's} Let $A$ and $B$ be C*-algebras with faithful
states $\phi_A$ and $\phi_B$ respectively. Consider the C*-algebra
reduced free product
$$(\gotA, \phi )= (A,\phi_A)* (B, \phi _B).$$
Suppose there is a partial isometry $v\in A$, whose range projection
$q=vv^*$, and domain projection $p=v^*v$, are orthogonal, and such
that, for some $0 < \lambda <1$, we have $\phi _A(vx)= \lambda
\phi_A(xv)$ for all $x\in A$. Let
$$A_{00} = \C p \oplus \C q \oplus (1-p-q)A(1-p-q)$$
and let $\gotAAA = C^*(A_{00}\cup B)$. Suppose that $q$ is
equivalent in the centralizer of the restriction of $\phi$ to
$\gotAAA$ to a subprojection of $p$, and that the centralizer of the
restriction of $\phi$ to $q\gotAAA q$ contains a unital abelian
subalgebra on which $\phi$ is diffuse. Suppose that $p$ is full in
$\gotA$.

Then $\gotA$ is simple and purely infinite.
\end{corollary}

\bigskip

\noindent {\it Proof of Corollary \ref{corol:Dykema's}.} We have to
show that the hypothesis of Theorem \ref{thm:Dykema's} are
satisfied. We take $P=1$, and $\Phi =\phi$.

By hypothesis, there exists an element $y$ in the centralizer of the
restriction of $\phi$ to $\gotAAA$ such that $yy^*=q$ and
$p_1:=y^*y\le p$. Note that $y\in (p+q)\gotAAA (p+q)$. By (a slight
extension of) \cite[Proposition 2.8]{DykemaScan}, the algebra
$(p+q)\gotA (p+q)$ is generated by $(p+q)\gotAAA (p+q)$ and
$(p+q)A(p+q)$, which are free with amalgamation over $p\C \oplus
q\C$. It follows that $y$ belongs to the centralizer of the
restriction of $\phi$ to $(p+q)\gotA (p+q)$. But since $y\in
(p+q)\gotA (p+q)$ and $\phi ((p+q)\gotA (1-p-q))= 0$, we get that
$y$ belongs to the centralizer of $\phi$ in $\gotA$.

Also by hypothesis, the centralizer of the restriction of $\phi$ to
$q\gotAAA q$ contains a unital abelian subalgebra $\mathcal D$ on
which $\phi$ is diffuse. Observe that, since $\phi (p\gotAAA q)=0$,
we get that $\mathcal D$ is contained in the centralizer of the
restriction of $\phi$ to $(p+q)\gotAAA (p+q)$. Now the same argument
as before shows that $\mathcal D$ is contained in the centralizer of
$\phi$ in $\gotA$.

It only remains to check that all the projections $p_k$ are
$A$-free. We will show by induction on $k$ that $p_k$ belongs to the
set \begin{equation} \label{eq:alternating} S: = p\C +\overline{\sum
_{i\ge 0} pB^{{\rm o}}(A^{{\rm o}}B^{{\rm o}})^ip} .
\end{equation}
This clearly shows that the projections $p_k$ are $A$-free.

Note that every element in $p\gotAAA p$ belongs to $S$. Now for
$k=1$, the result is clear since $p_1=y^*vv^* y=y^*qy\in p\gotAAA
p$. Assume that $p_k$ belongs to $S$, and let us show that
$p_{k+1}\in S$. Note that $y^*v$ belongs to the closed linear span
of
$$\sum _{i\ge 0} pB^{{\rm o}}(A_{00}^{{\rm o}}B^{{\rm o}})^iv .$$
Observe now that $p_{k+1}= (y^*v)p_k(v^*y)$, so we are led to
consider either terms of the form
$$(pB^{{\rm o}}(A_{00}^{{\rm o}}B^{{\rm o}})^iq)vv^*(qB^{{\rm o}}(A_{00}^{{\rm o}}B^{{\rm o}})^jp)
$$
or terms of the form
$$(pB^{{\rm o}}(A_{00}^{{\rm o}}B^{{\rm o}})^iv) (pB^{{\rm o}}A^{{\rm o}}\cdots A^{{\rm o}}B^{{\rm o}}p)(v^*B^{{\rm o}}(A_{00}^{{\rm o}}B^{{\rm o}})^jp) .$$
In the former case we get an element in $p\gotAAA p\subset S$. In
the latter case, since $v\in A^{{\rm o}}$, we get an element in
$pB^{{\rm o}}(A^{{\rm o}}B^{{\rm o}})^lp$ for some $l\ge 1$. So, in
either case, we obtain an element in $S$, as desired. \qed

\bigskip

\noindent {\it Proof of Theorem \ref{thm:Dykema's}.} The proof
follows the steps of the one of \cite[Theorem 3.1]{DykemaScan}. Note
that there was an error in the statement and application of
\cite[Theorem 2.1(i)]{DykemaRordamCJM}. The word ``outer'' that
appears there should be ``multiplier outer,'' i.e., outer relative
to the multiplier algebra of $A$, instead of relative to the
unitization of $A$. This led to deficiencies in the proof of
\cite[Theorem 3.1]{DykemaScan}, which have been corrected in
\cite{ADR}. This involves a change in the order of the different
steps of that proof. Here we will outline the main steps of the
proof of our result, referring to \cite{DykemaScan} for the proofs
which are identical.

\medskip

Recall that $w=y^*v$. Since $\phi (vx)=\lambda \phi (xv)$ and $\phi
(y^*x)=\mu \phi (xy^*)$ for all $x\in \gotA$, we have $\phi (wx)=
(\lambda \mu)\phi (xw)$ for all $x\in \gotA$. It follows that
$$\phi (p_k)=(\lambda\mu)^k\phi (p),$$
so that, recalling that $\lambda \mu <1$, we have that  $\lim_{k\to
\infty} \phi (p_k)=0$.

\medskip

Since $p$ is full in $\gotA$, we need only to show that $p\gotA p$
is purely infinite and simple. By assumption $q\gotAA q$ contains a
unital diffuse abelian subalgebra $\mathcal D$ which is contained in
the centralizer of $\phi$ in $\gotA$. Observe that
$$q\gotAA q=yy^*\gotAA yy^*=y(y^*\gotAA y)y^*\subseteq y\gotAA y^*\subseteq q\gotAA q  \, ,$$
so that $q\gotAA q=y\gotAA y^*$, and $p_1\gotAA p_1 =y^*y \gotAA
y^*y =y^*(y\gotAA y^*)y =y^*q\gotAA q y$. It follows that
$y^*\mathcal D y $ is a unital diffuse abelian subalgebra of
$p_1\gotAA p_1$ which is contained in the centralizer of $\phi$,
since for $d\in \mathcal D$ and $a\in \gotA$ we have:
$$\phi ((y^*dy)a)=\mu \phi (dyay^*)= \mu \phi (yay^* d)= \mu \mu ^{-1}
\phi (ay^*dy)=\phi (a(y^*dy)).$$ Therefore $p_1\gotAA p_1$ contains
a unital diffuse abelian subalgebra which is contained in the
centralizer of $\phi$ in $\gotA$.

\medskip

Note that $A$ is generated by $A_0\cup \{v\}$, because
$pA(1-p)=v^*vA(1-p)\subseteq v^*(1-p)A(1-p)$. It follows that
$\gotA=C^*( \gotAA \cup \{w\})$, and thus $p \gotA p=C^*(p\gotAA p
\cup \{w\})$.

\medskip

Note that since $p\le P$ and $PC=P\C$, we have
$$(p\gotA p)^{{\rm o}}=\{x\in
p\gotA p\mid \Phi (x)=0\}=\{x\in p\gotA p\mid \phi (x)=0\}.$$

\medskip

Let $\Theta$ be the set of all
$$x= x_1x_2\cdots x_n\in \Lambda^{{\rm o}} ((p\gotAA p)^{{\rm o}},\{w^k\mid k\ge 1\}\cup \{(w^*)^k\mid k\ge 1\})$$
such that whenever $2\le j\le n-1$  and $x_j \in (p\gotAA p)^{{\rm
o}} $:
\begin{align*}
&\text{ if } \, x_{j-1}= w^i \, (i>0) \text{ and } x_{j+1}= (w^*)^j
\, (j>0)\, \text{ then }\, x_j\in p\gotAA p\ominus pA_0p\\ & \text{
if } \, x_{j-1}= (w^*)^i \, (i>0) \text{ and } x_{j+1}= w^j \,
(j>0)\, \text{ then }\, x_j\in p_1\gotAA p_1\ominus y^*(qA_0q)y.
\end{align*}
(Note that $E$ restricts to a conditional expectation $p\gotAA p\to
pA_0p$ and $y^*E(y\cdot y^*)y$ provides one from $p_1\gotAA p_1$
onto $y^*(qA_0q)y$.) Note that $\text{span} (\{p\}\cup \Theta )$ is
the $*$-algebra generated by $p\gotAA p \cup \{w\}$ (see
\cite{DykemaScan}).

\medskip

For $x\in \Theta$ of length $\ell (x)$ and $0\le j \le \ell (x)$,
let $t_j(x)$ be the number of $w$ minus the number of $w^*$
appearing in the first $j$ letters of $x$. (Here, of course,
$t_0(x)=0$.) For an interval $I$ of $\Z$ containing $0$, define
$$\Theta _I =\{ x\in \Theta \mid t_{\ell (x)} =0 \text{ and } \forall 1\le j\le \ell (x), t_j(x)\in I\}.$$
Then $\text{span} (\Theta _I \cup \{p\})$ is a $*$-subalgebra of
$p\gotA p$. Let $\gotA _I=\ol{\text{span}} (\Theta _I \cup \{p\})$.

\medskip

\noindent {\it Claim 1:} $E(x)=0$ for all $x\in \Theta
_{(-\infty,\infty)}\setminus (p\gotAA p)^{{\rm o}}$ and $\phi (x)=
0$ for all $x\in \Theta _{(-\infty,\infty)}$.

\medskip

\noindent {\it Proof of Claim 1.} See the proof of \cite[Claim
3.3]{DykemaScan}. \qed

\medskip

\noindent {\it Claim 2:} The subalgebras $w^*\gotA _{(-\infty,0]}w$
and $\gotA _{[0,\infty)}$ are free with amalgamation over $pA_0p$
(with respect to the restrictions of the conditional expectation
$E$).

\medskip

\noindent {\it Proof of Claim 2.} See the proof of \cite[Claim
3.4]{DykemaScan}. \qed

\medskip

\noindent {\it Claim 3:} $\gotA _{(-\infty, 0]}$ is simple.

\medskip

\noindent {\it Proof of Claim 3.} $\gotA _{(-\infty , 0]}$ is
generated by $w^*\gotA _{(-\infty, 0]}w $ and $p\gotAA p$, which by
Claim 2 are free with amalgamation over $pA_0p$. Let $\gotAAp
=C^*(B\cup (\C p+(1-p)A(1-p)))$. Then $p\gotAA p$ is the C*-algebra
generated by $pA_0p$ and $p\gotAAp p$, which are free with respect
to $\phi$ (cf. \cite[2.8]{DykemaTAMS}). By using Claim 2 we get that
$\gotA _{(-\infty , 0]}$ is generated by $w^*\gotA _{(-\infty ,
0]}w$ and $p\gotAAp  p$, which are free with respect to $\phi$ (with
amalgamation over $p\C$).

But now $w^*\gotA _{(-\infty , 0]}w$ contains $w^*\gotAA w = w^*
p_1\gotAA p_1w$ and, by the same argument applied above to
$p_1\gotAA p_1$, we can deduce that $w^* p_1\gotAA p_1w $ contains a
unital diffuse abelian subalgebra, which is contained in the
centralizer of $\phi $ in $\gotA $. So \cite[3.2]{DykemaTAMS} gives
that $\gotA _{(-\infty , 0]}$ is simple.\qed

\medskip

\noindent {\it Claim 4:} For all $n\ge 0$, the C*-algebra $\gotA
_{(-\infty ,n]}$ is simple.

\medskip

\noindent {\it Proof of Claim 4.} Use the same proof as in
\cite[proof of Claim 3.6]{DykemaScan}.\qed

\medskip

\noindent {\it Claim 5:} Let $n\ge 0$, $k\ge 1$. Then $p_{n+1}\gotA
_{(-\infty , n]} p_{n+1}$ and $\{p_{n+k}\}$ are free (with
amalgamation over $\C p_{n+1}$) with respect to $\phi$ (after
scaling).

\medskip

\noindent {\it Proof of Claim 5.} As in \cite[proof of Claim
3.7]{DykemaScan}, we may reduce to show that $w^*\gotA _{(-\infty ,
0]}w$ and $\{p_{k-1}\}$ are free (with amalgamation over $\C p$).
Obviously we can assume that $k\ge 2$.

Recall that, by hypothesis, $p_{k-1}$ is $A$-free, so that
$E(p_{k-1})\in \C p$. Set $b:= p_{k-1}-\frac{\phi (p_{k-1})}{\phi
(p)}p$. Then $b$ belongs to $\gotA _{[0,\infty)}\ominus pA_0p$,
because $E(b)=0$. We have to show that $\phi (x)=0$ for all $x\in
\Lambda^{{\rm o}}((w^*\gotA _{(-\infty, 0]}w) ^{{\rm o}}, \{b\})$,
Since $w^*\gotA _{(-\infty, 0]}w$ and $\gotA_{[0,\infty )}$ are free
with respect to $E$, with amalgamation over $pA_0p$ (Claim 2), we
are led to show that the words in $\Lambda^{{\rm o}}((pA_0p)^{{\rm
o}},\{b\})$ of length $>1$ belong to the kernel of $E$ (and so they
belong to $\gotA _{[0,\infty)}\ominus pA_0p$). But this follows from
the $A$-freeness hypothesis of $p_{k-1}$.\qed

\medskip

There exists an injective endomorphism $\sigma \colon \gotA
_{(-\infty, \infty)}\to \gotA _{(-\infty, \infty)}$ given by $\sigma
(a)= waw^*$. Since $p\gotA p=C^*(\gotA _{(-\infty ,\infty)}\cup
\{w\})$, $p\gotA p$ is a quotient of $\gotA
_{(-\infty,\infty)}\rtimes_{\sigma} \N$. It is enough thus to show
that  $\gotA _{(-\infty,\infty)}\rtimes_{\sigma} \N$ is simple and
purely infinite.

\medskip

\noindent {\it Claim 6:} For all $m\ge 1$, $\alpha ^m$ is multiplier
outer in $\overline{\gotA} _{(-\infty,\infty)}$, where
$\overline{\gotA} _{(-\infty,\infty)}$ denotes the inductive limit
$\lim (\gotA _{(-\infty,\infty)}\overset{\sigma}{\longrightarrow
}\gotA _{(-\infty,\infty)}\overset{\sigma}{\longrightarrow} \cdots
)$

\medskip

\noindent {\it Proof of Claim 6.} This is proved in Lemma 2.3 of
\cite{ADR}. \qed

\medskip

\noindent {\it Claim 7:} Let $D$ be a nonzero hereditary
C*-subalgebra of $\gotA _{(-\infty,\infty)} $. Then there is a
projection in $D$ that is equivalent in $\gotA _{(-\infty, \infty)}$
to $p_n$ for some $n$.

\medskip

The proof of Claim 7 is exactly the same as the corresponding one in
\cite[proof of Claim 3.8]{DykemaScan}.

\medskip

Claim 6 and Claim 7 show that the hypothesis in \cite[Theorem
2.1(ii)]{DykemaRordamCJM} are satisfied in our situation, and so we
get from this result that $\gotA _{(-\infty,\infty)}\rtimes_{\sigma}
\N$ is simple and purely infinite. (Recall that the hypothesis of
$\alpha ^m$ being outer in \cite{DykemaRordamCJM} must be
interpreted as being multiplier outer.) \qed

\section{Purely infinite simple reduced graph C*-algebras}
\label{sect:pur-infinite}

We first give an application of Dykema's result (Corollary
\ref{corol:Dykema's}) to reduced free products of Cuntz algebras. It
was shown in \cite[Proposition 4.2]{AG2} that the C*-algebras
$\Cstred (E,C)$ in the next proposition are simple. By using
Dykema's Theorem, we can now show that they are also purely
infinite.

\begin{proposition}
\label{prop:one-vertex} Let $n,m>1$, and let $(E,C)$ be the
separated graph with one vertex $v$ and with $C_v :=\{X,Y\}$, where
$|X|=n$ and $|Y|=m$. Then the reduced graph C*-algebra
$\Cstred(E,C)$ is purely infinite and simple.
\end{proposition}

\begin{proof}
Set $A :=\calO_n$ and $B :=\calO _m$, where as usual $\mathcal O_k$
denotes the Cuntz algebra, and identify $A=C^*(E_X)$ and
$B=C^*(E_Y)$. Then $(\Cstred(E,C), \phi)$ is the reduced free
product of $(\calO_n,\phi_n)$ and $(\calO _m,\phi _m)$, where we
denote by $\phi _k$ the canonical faithful state on $\calO_k$ (see
\cite[Theorem 2.1]{AG2}). Set $X=\{e_1,\dots ,e_n\}$, $Y=\{
f_1,\dots ,f_m\}$, and $p_i=e_ie_i^*$, $r_j=f_jf_j^*$, for $1\le
i\le n$, $1\le j\le m$. Observe that $\phi_n(p_i)=1/n$ and
$\phi_m(f_j)=1/m$. Set $p=p_1$, $q=e_2e_1e_1^*e_2^*$, $v=
e_2e_1e_1^*$. Then $vv^*= q$ and $v^*v=p$, and moreover $\phi _n
(vx)=\frac{1}{n}\phi _n (xv)$ for all $x\in A$. Clearly $p$ is full
in $\Cstred(E,C)$. Note that the centralizer of $\phi_m$ in $B=\calO
_m$ contains an abelian subalgebra $D$ on which $\phi_m$ is diffuse,
namely the diagonal C*-subalgebra $D$ generated by all the elements
$\lambda \lambda^*$, with $\lambda $ a path in $E_Y$. (The spectrum
of this algebra is a Cantor set.) So the rest of the conditions
needed to apply Corollary \ref{corol:Dykema's} is verified in the
same way as in \cite[3.9(iii)]{DykemaScan}.
\end{proof}



We now recover the setting of the introduction. Recall that an
abelian monoid $M$ is said to be {\it conical} in case, for $x,y\in
M$, $x+y=0$ implies $x=y=0$.
 Let $F$ be the free
abelian monoid on free generators $a_1,a_2,\dots ,a_n$. Let
$$x=\sum _{i=1}^n
r_ia_i,\qquad y =\sum_{i=1}^n s_i a_i $$ be nonzero elements in $F$.
Let $M$ be the abelian conical monoid ${F/\sim}$, where $\sim$ is
the congruence on $F$ generated by $(x,y)$. We shall assume, without
loss of generality, that $r_i+s_i>0$ for $i=1,\dots ,n$. (Otherwise
the C*-algebras we consider will have a finite-dimensional direct
summand.)

Let $(E,C)$ be the separated graph associated to the presentation
$\langle a_1,a_2,\dots ,a_n \mid x=y\rangle$. To be precise $E$ is a
graph with $n+1$ vertices $E^0=\{v,w_1, w_2,\dots ,w_n\}$ and $N+M$
arrows, where $M=\sum _{i=1}^n r_i$ and $N=\sum _{i=1}^n s_i$, with
$v$ being a source and all $w_i$ being sinks. The arrows in $E$ are
labeled as $\alpha ^{(i)}_j$, for $1\le j\le s_i$, $1\le i\le n$;
and $\beta^{(i)}_j$, for $1\le j\le r_i$, $1\le i\le n$. We have
$r(\alpha _j^{(i)})= r(\beta_k^{(i)})=w_i$, and $s(\alpha _j^{(i)})=
s(\beta_k^{(i)})=v$.
 There are
two elements $X,Y$ in $C=C_v$, given by
$$X=\{\alpha ^{(i)}_j\mid 1\le
j\le s_i, 1\le i\le n\},\qquad Y =\{ \beta^{(i)}_j \mid 1\le j\le
r_i, 1\le i\le n\} \, .$$ The C*-algebra $\Cstred (E,C)$ turns out
to be the amalgamated free product
$$(\Cstred (E,C), \Phi )= (A, \Phi_A)*_{\C^{n+1}} (B,\Phi _B)\, ,$$
where $A= \prod _{i=1}^n M_{s_i+1}(\C)$ and $B=\prod_{i=1}^n
M_{r_i+1}(\C)$. We will denote the canonical matrix units in $A$ and
$B$ by $e_{jk}^{(i)}$,  $0\le j,k \le s_i$, $1\le i\le n$, and
$f^{(i)}_{jk}$, $1\le j,k\le r_i$, $1\le i\le n$, respectively. With
this notation we can describe the unital embeddings $\iota _A$ and
$\iota _B$ of $\C^{n+1}$ into $A$ and $B$ by the formulas
$$\iota _A(e_i)=e^{(i)}_{00},\qquad \iota _B(e_i)= f^{(i)}_{00}$$
for $i=1,\dots ,n$, where $e_i$, $i=1,\dots ,n+1$ are the minimal
projections of $\C^{n+1}$. The conditional expectation $\Phi _A$ is
given as follows
$$\Phi _A \Big( \sum _{i=1}^n\sum _{j,k=0}^{s_i} a^{(i)}_{jk}e^{(i)}_{jk} \Big)=
\Big( a^{(1)}_{00}, a^{(2)}_{00},\dots ,a^{(n)}_{00},
\frac{1}{N}\sum _{i=1}^n\sum _{j=1}^{s_i}a_{jj}^{(i)}\Big)\, ,$$
with a similar formula holding for $\Phi _B$.

In several cases we can show that $\gotA$ is simple, using a
generalization of Avitzour's Theorem (\cite[4.3]{AG2}) (see
\cite{avitzour} for the original result on free products).

\begin{lemma}
\label{lem:another-Aznavour} Let $(E,C)$ be the separated graph
associated to the presentation $\langle a_1,\dots ,a_n\mid \sum
_{i=1}^n r_ia_i=\sum _{i=1}^n s_ia_i\rangle $, as described above.
Assume that $M\ge 2$ and $ N\ge 3$. Then the C*-algebra $\Cstred
(E,C)$ is simple.
\end{lemma}

\begin{proof}
We shall use \cite[4.3 and 4.4]{AG2}. Consider the projection in
$\Cstred (E,C)$ corresponding to the vertex $v$, which is denoted by
the same symbol. This projection corresponds to $(0,\dots ,0,1)\in
\C^{n+1}$ in the above picture. Observe that
$$vAv\cong \prod_{i=1}^n
M_{s_i}(\C)\quad  \text{and} \quad vBv\cong \prod_{i=1}^n
M_{r_i}(\C),$$ and the canonical conditional expectations induce the
tracial states $\tau_{vAv}$ and $\tau _{vBv}$ on $vAv$ and $vBv$
given by
$$\tau_{vAv} (x_1,x_2,\dots ,x_n )=\frac{1}{M}\sum _{i=1}^n
\text{Tr}_{r_i}(x_i),\qquad \tau_{vBv} (x_1,x_2,\dots ,x_n
)=\frac{1}{N}\sum _{i=1}^n \text{Tr}_{s_i}(x_i)$$ respectively. Let
$D_A$ and $D_B$ denote the canonical maximal commutative subalgebras
of $vAv$ and $vBv$. Then $\text{dim}_{\C}(D_A)=M\ge 2$ and
$\text{dim}_{\C}(D_B)=N\ge 3$. We can thus find unitaries $a$ in
$D_A$, and $b,c$ in $D_B$ such that
$$\tau_{vAv} (a) = 0,\qquad \tau_{vBv} (b)=\tau_{vBv} (c)= \tau_{vBv} (b^*c)= 0 .$$
These unitaries satisfy all the hypothesis required in
\cite[Proposition 4.3]{AG2}.

In order to show simplicity, it remains to observe that $v$ is full
in $\gotA$. To see this, it is enough to show that $w_i\precsim v$
in $\gotA$ for all $i$. Now given $i\in \{1,\dots ,n\}$, either
$r_i\ne 0$ or $s_i\ne 0$ by the hypothesis that $r_i+s_i>0$, so
either $w_i\precsim v$ in $A$ or $w_i\precsim v$ in $B$. In any case
we have $w_i\precsim v$ in $\gotA$, as wanted. We can therefore
conclude from \cite[Corollary 4.4]{AG2} that $\Cstred (E,C)$ is
simple.
\end{proof}

\begin{theorem}
\label{thm:motivating} Let $(E,C)$ be the separated graph associated
to the presentation $\langle a_1,\dots ,a_n\mid \sum _{i=1}^n
r_ia_i=\sum _{i=1}^n s_ia_i\rangle $, as described above, and put
$M=\sum _{i=1}^n r_i$ and $N=\sum _{i=1}^n s_i$. Assume that there
is $i_0\in \{1,\dots ,n\}$ such that  $s_{i_0}\ge 1$ and $r_{i_0}\ge
1$, and that $2\le M<N$. Then the C*-algebra $\Cstred (E,C)$ is
purely infinite simple.
\end{theorem}

\begin{proof}
Write $\gotA =\Cstred (E,C)$. We have $(\gotA,
\Phi)=(A,\Phi_A)*_{\C^{n+1}} (B, \Phi_B)$, as described above. By
Lemma \ref{lem:another-Aznavour}, $\gotA$ is a simple C*-algebra.
Therefore every nonzero projection in $\gotA$ is full.

\smallskip

By hypothesis, there exists $i_0\in \{1,\dots ,n\}$ such that
$s_{i_0}\ge 1$ and $r_{i_0}\ge 1$. Without loss of generality, we
shall assume that $i_0=1$.

\smallskip

Now consider the faithful state $\gamma $ on $\C^{n+1}$ given by
$$\gamma (x_1,x_2,\dots ,x_{n+1})= \frac{1}{n+1}\sum _{i=1}^{n+1} x_i\, ,$$ and write $\phi
=\gamma \circ \Phi$, which is a faithful state on $\gotA$.

Let $P=(1,0,\dots ,0)\in \C^{n+1}$, and consider the projections
$p=e^{(1)}_{00}$ and $q= e^{(1)}_{11}$ in $A$. Observe that $p=P$ in
$\gotA$. Let $v=e^{(1)}_{10}\in A$, and observe that $v$ is a
partial isometry in $A$ such that $vv^*=q$, $v^*v=p$, and $\phi
(vx)=\lambda \phi (xv)$, where $\lambda = 1/N$.

Set $A_0=pAp+(1-p)A(1-p)\cong \C\times M_{s_1}(\C)\times
\prod_{i=2}^n M_{s_i+1}(\C)$, and set $\gotAA=C^*(A_0\cup B)$. We
also put $A_{00}=\C p+\C q + (1-p-q)A(1-p-q)$, and $\gotAAA=
C^*(A_{00}\cup B)$. Let $\gotB$ be the C*-subalgebra of
$(1-p)\gotAAA (1-p)$ generated by $\{e^{(i)}_{jj}, f^{(i)}_{kk}\mid
1\le j\le s_i, 1\le k\le r_i, i=1,\dots ,n \}$. Note that we have a
natural isomorphism
$$(C_r^*(\Z_N*\Z_M),\tau) \longrightarrow (\gotB , (n+1)\phi)$$
sending the canonical spectral projections in $C_r^*(\Z_N)$ to
$e^{(i)}_{jj}$, $1\le j\le s_i$, $1\le i\le n$, and the canonical
spectral projections in $C_r^*(\Z_M)$ to $f^{(i)}_{kk}$, $1\le k\le
r_i$, $1\le i\le n$. Since
$$\phi (q)= \frac{1}{(n+1)N} <\frac{1}{(n+1)M} =\phi (f^{(1)}_{11}) \, ,$$
we get that $[q]<[f_{11}^{(1)}]$ in $K_0(\gotB)$ by \cite[Theorem
2]{DykemaRordamGFA}. Since $M\ge 2$ and $N\ge 3$, it follows from
\cite[Corollary 3.9]{DyHaagRor} that  $C^*_r(\Z_N*\Z_M)$ has stable
rank one. Therefore we get that $q\precsim f^{(1)}_{11}$ in $\gotB$.
Observe that $\gotB$ is contained in the centralizer of $\phi$, so
$q=zz^*$ and $z^*z\le f^{(1)}_{11}$ for some $z$ in the centralizer
of $\phi$. Now consider $y= zf^{(1)}_{10}\in \gotAAA$. We have
$yy^*=q$ and $y^*y\le f^{(1)}_{00}=e^{(1)}_{00}=p$. Moreover, since
$z$ belongs to the centralizer of $\phi$, we have
$$\phi (y^*x)=\mu \phi (xy^*) \, $$
for all $x\in \gotA$,  where $\mu =M$. Note that $\lambda \mu
=\frac{M}{N}<1$. Set $w=y^*v$ and $p_k=w^k(w^*)^k$ for all $k\ge 1$.
Then, since $y\in \gotAAA$, the same proof as in Corollary
\ref{corol:Dykema's} gives that the projections $p_k$ are $A$-free.

\smallskip

It remains to check that $q\gotAAA q $ contains a unital diffuse
abelian subalgebra contained in the centralizer of $\phi$. For this
it is enough to see that $qC^*_r (\Z_N* \Z_M)q\cong q\gotB q$
contains a unital diffuse abelian subalgebra, where $q$ is
identified with a minimal projection in $C^*_r(\Z_N)$. Let $r$ be a
minimal projection in $C^*_r(\Z_M)$. Let $\gotB '$ be the
C*-subalgebra of $C^*_r (\Z_N * \Z_M)=(C^*_r(\Z_N),
\tau_N)*(C^*_r(\Z_M),\tau_M )$ generated by $q, 1-q, r, 1-r$ (where
$\tau _i$ is the canonical trace on $C^*_r(\Z_i)$). By
\cite[2.7]{DykemaTAMS}, we have
$$(\gotB ', \tau |_{\gotB '} )= \overset{(1-q)\wedge
r}{\underset{\frac{N-M}{NM}}{\C}} \oplus C([a,b], M_2(\C))\oplus
\overset{(1-q)\wedge (1-r)}{\underset{
1-\frac{1}{N}-\frac{1}{M}}{\C}} ,$$ for some $0<a<b<1$. In this
picture $q$ corresponds to the projection $0\oplus
\begin{pmatrix}
0 & 0 \\ 0 & 1
\end{pmatrix}\oplus 0$ and $\tau$ is given by the indicated weights on the projections
$(1-q)\wedge r$ and $(1-q)\wedge (1-r)$, together with an atomless
measure whose support is $[a,b]$. It follows that $q\gotB 'q$
contains a unital diffuse abelian subalgebra, and the same will be
true for $q\gotB q$.
\end{proof}

\begin{remarks}
\label{rem:Cuntzalgsandothers}

\noindent (i) We remark that the above proof, combined with the
proof of Theorem \ref{thm:Dykema's}, gives the well-known
description of $\mathcal O _n $ as a crossed product when applied to
the presentation $\langle a \mid a=na \rangle$.

\noindent (ii) Theorem \ref{thm:motivating} does not hold if the
hypothesis that there is $i_0\in \{1,\dots n\}$ such that
$r_{i_0}>0$ and $s_{i_0}>0$ is suppressed, see Proposition
\ref{prop:capequalzero}.
\end{remarks}

It is worth to state explicitly the following particular case.
Recall that the C*-algebras $\Cstred (E(m,n), C(m,n))$ provide
higher dimensional generalizations of Cuntz algebras. Indeed, we
have $\Cstred (E(1,n), C(1,n))\cong M_2(\mathcal O _n )$ (see
\cite[Example 4.5]{AG2}). The representation $\mathcal O _2\cong
M_2(\mathcal O _2)\cong M_2(\C)*_{\C^2}M_{3}(\C)$ goes back to Choi
(\cite[Theorem 2.6]{Choi}).

\begin{corollary}
\label{cor:Umn} Assume that $1\le m<n$. Then the C*-algebra $\Cstred
(E(m,n), C(m,n))$ is purely infinite simple.
\end{corollary}

\section{The finite case}
\label{sect:finite-case}

We will use the following well-known result for the existence of
tracial states on an amalgamated free product.

\begin{lemma}
\label{lem:tracesongotA} Let $(\gotA,\Phi)= (A,\Phi_A)*_C(B,\Phi
_B)$ be an amalgamated free product with respect to faithful
conditional expectations $\Phi_A$ and $\Phi _B$. Then there is a
faithful tracial state on $\gotA$ if and only if there is a faithful
state $\tau$ on $C$ such that both $\tau \circ \Phi _A$ and $\tau
\circ \Phi _B$ are tracial states on $A$ and $B$ respectively. In
this case the tracial state on $\gotA$ is defined by $\theta= \tau
\circ \Phi$.
\end{lemma}

We now show the existence of a faithful tracial state in the
balanced case, as follows:

\begin{proposition}
\label{prop:balancedcase} Let $(E,C)$ be the separated graph
associated to the presentation $\langle a_1,\dots ,a_n\mid \sum
_{i=1}^n r_ia_i=\sum _{i=1}^n s_ia_i\rangle $, and put $M=\sum
_{i=1}^n r_i$ and $N=\sum _{i=1}^n s_i$. Assume that $N=M$. Then the
C*-algebra $\Cstred (E,C)$ has a faithful tracial state. In
particular it follows that $\Cstred (E,C)$  is stably finite.
Moreover, if in addition $N=M>2$, then $\Cstred (E,C)$ is simple and
has a unique tracial state.
\end{proposition}

\begin{proof} We put $\gotA = \Cstred (E,C)$ and use the notation introduced in Section
\ref{sect:pur-infinite}. Define the faithful state $\tau $ on $
\C^{n+1}$ by $$\tau (a_1,a_2,\dots ,a_{n+1})= \frac{1}{N+n}\sum
_{i=1}^n a_i+ \frac{N}{N+n} a_{n+1}.$$ Since $N=M$, it is easily
checked that $\tau \circ \Phi_A$ and $\tau \circ \Phi_B$ are tracial
states of $A$ and $B$ respectively. So $\theta = \tau \circ \Phi$ is
a faithful tracial state on $\gotA$ by Lemma \ref{lem:tracesongotA}.
\end{proof}

We now study the case where $\{i\in \{1,\dots ,n\}\mid r_i\ne
0\}\bigcap \{i\in \{1,\dots ,n\}\mid s_i\ne 0\}=\emptyset$. This
case corresponds (by a Morita-equivalence) to an ordinary free
product of finite-dimensional C*-algebras, with respect to faithful
tracial states.

\begin{proposition}
\label{prop:capequalzero} Assume that $\{i\in \{1,\dots ,n\}\mid
s_i\ne 0\}\bigcap \{i\in \{1,\dots ,n\}\mid r_i\ne 0\}=\emptyset$.
Then $\Cstred (E,C)$ admits a faithful tracial state. Moreover
$\Cstred (E,C)$ is simple if $\, 2\le M<N$.
\end{proposition}

\begin{proof}
Set $I_1=\{i\in \{1,\dots , n\}\mid s_i>0\}$ and $I_2=\{ i\in
\{1,\dots ,n\}\mid r_i >0\}$. Then, by hypothesis $\{1,\dots ,n\}$
is the disjoint union of $I_1$ and $I_2$. Set $n_i:=|I_i|$ for
$i=1,2$, and $K:= n_1M+n_2N+NM$, where as usual $N=\sum _{i=1}^n
s_i$ and $M=\sum _{i=1}^n r_i$. Define a faithful state $\tau$ on
$\C^{n+1}$ by
$$\tau (a_1,\dots ,a_n, a_{n+1})= \frac{1}{K} (M \sum _{i\in I_1} a_i+N\sum _{i\in I_2} a_i+
NMa_{n+1}).$$ Then $\tau \circ \Phi_A$ and $\tau \circ \Phi_B$ are
tracial states on $A$ and $B$ respectively. We check this for
$\tau\circ \Phi_A$:
\begin{align*}
(\tau \circ \Phi_A)\Big( \sum _{i=1}^n \sum _{j,k=0}^{s_i}
a^{(i)}_{jk}e^{(i)}_{jk} \Big) & = \tau \Big( a_{00}^{(1)},\dots
,a_{00}^{(n)},\frac{1}{N}\sum _{i\in I_1}\sum _{j=1}^{s_i}
a_{jj}^{(i)} \Big)\\
& = \frac{1}{K} \Big( M\sum _{i\in I_1} a_{00}^{(i)} +N\sum _{i\in
I_2}
a_{00}^{(i)}+ \frac{NM}{N} \sum _{i\in I_1}\sum _{j=1}^{s_i} a_{jj}^{(i)}\Big) \\
& = \frac{M}{K} \sum _{i\in I_1} \sum _{j=0}^{s_i} a_{jj}^{(i)} +
\frac{N}{K} \sum _{i\in I_2} a_{00}^{(i)}\, ,
\end{align*}
which is a trace on $A$. Similarly $\tau \circ \Phi_B$ is a trace on
$B$. By Lemma \ref{lem:tracesongotA}, it follows that $\tau \circ
\Phi$ is a faithful trace on $\Cstred (E,C)$.
\end{proof}

We can now provide the proof of theorem \ref{thm:main}:

\medskip

\noindent {\it Proof of Theorem \ref{thm:main}} Assume that $2\le
M\le N$. Set $I_1=\{i\in \{1,\dots , n\}\mid s_i>0\}$ and $I_2=\{
i\in \{1,\dots ,n\}\mid r_i >0\}$.  If $M<N$ and $I_1\cap I_2\ne
\emptyset$, then  $\Cstred (E,C)$ is purely infinite simple by
Theorem \ref{thm:motivating}. If $M<N$ and $I_1\cap I_2=\emptyset$,
then $\Cstred (E,C)$ admits a faithful tracial state by Proposition
\ref{prop:capequalzero}. If $N=M$, then $\Cstred (E,C)$ admits a
faithful tracial state by Proposition \ref{prop:balancedcase}.

\medskip

Assume now that $2\le M\le N$, $\, N+M\ge 5$ and $\Cstred (E,C)$ is
finite. Then $\gotA$ is simple by Lemma \ref{lem:another-Aznavour}
and by the dichotomy showed before, there exists a tracial state on
$\gotA$. If $\phi_1$ and $\phi_2$ are two tracial states on $\gotA
:= \Cstred (E,C)$, then since, by \cite[Proposition 4.3]{AG2}
$$\Phi (x)\in {\rm \overline{co}} \{ u^*xu: u \text{ unitary in } v\gotA v\}$$
for all $x\in v\gotA v$, we get that $\phi_1$ and $\phi_2$ agree on
$v\gotA v$. Since $v$ is a full projection in $\gotA$, we obtain
that $\phi_1=\phi_2$. Thus there is exactly one tracial state on
$\gotA$, as desired. \qed

\medskip

In order to have a complete description of the graph C*-algebras
$\Cstred (E,C)$ in the one-relator case, it remains to study the
cases where $M=1$, and also the cases where $N=M=2$ in terms of
simplicity and uniqueness of the trace. All cases are easy to
analyze, except two. We first collect the easy cases in a lemma.

\begin{lemma}
\label{lem:easy-cases} Let $(E,C)$ be the separated graph associated
to the presentation $\langle a_1,\dots ,a_n\mid \sum _{i=1}^n
r_ia_i=\sum _{i=1}^n s_ia_i\rangle $, and put $M=\sum _{i=1}^n r_i$
and $N=\sum _{i=1}^n s_i$.
\begin{enumerate}
\item If $M=1$, then we have $\Cstred (E,C)\cong M_2(C^*(F))$, where $F$
is the graph obtained from $E$ by collapsing $v$ and $r(\beta_1)$
and eliminating the arrow $\beta _1$.
\item If $M=N=2$, then there is a faithful tracial state on $\Cstred (E,C)$,
and there are several cases:
\begin{enumerate}
\item  If $n=4$, then $\Cstred (E,C) $ is Morita-equivalent to
$C^*_r(\Z_2*\Z_2)\cong C^*(\Z_2*\Z_2)$, and so it is non-simple.
\item If $n=3$, and either $r(\alpha_1)=r(\alpha _2)$ or $r(\beta_1)=r(\beta_2)$, then
$\Cstred (E,C)$ is Morita-equivalent to $(M_2(\C), {\rm
Tr}_2)*\Big(\underset{1/2}{\C} \oplus \underset{1/2}{\C}\Big)$. It
is simple with a unique trace.
\item If $n=2$ and $r(\alpha_1)=r(\alpha_2)\ne
r(\beta_1)=r(\beta_2)$, then $\Cstred (E,C)$ is Morita-equivalent to
$(M_2(\C), {\rm Tr}_2)*(M_2(\C), {\rm Tr}_2)$. It is simple with a
unique trace.
\item If $n=2$ and $r(\alpha _1)\ne r(\alpha_2) $ and
$r(\beta_1)=r(\beta_2)$, then $\Cstred (E,C)$ is also simple with a
unique trace.
\item If $n=1$, then $\Cstred (E,C)$ is simple with a unique trace.
\end{enumerate}
\end{enumerate}

\end{lemma}

\begin{proof}
(1) It is a straightforward computation.

(2) There is a faithful trace by Proposition
\ref{prop:balancedcase}. In cases (b), (c), (d), (e) one can use
\cite[Corollary 4.4]{AG2} to show simplicity and uniqueness of the
trace, because both $vAv$ and $vBv$ are at least $2$-dimensional,
and at least one of them is a matrix algebra of size at least
$2\times 2$.
\end{proof}

\medskip

The two cases remaining to analyze when $N=M=2$ are:

$\bullet$ $n=3$ and $r(\alpha _i)=r(\beta_j)$ for some $i,j$.

$\bullet$ $n=2$ and $r(\alpha_i)=r(\beta_i)$ for $i=1,2$.

\medskip

We now study these two cases, in which we cannot apply the
generalization of Avitzour's Theorem. It is enough to pay attention
to the structure of the full corner $v\Cstred (E,C)v$. This is what
turns out to be a significant C*-algebra in these examples.

\begin{center}{
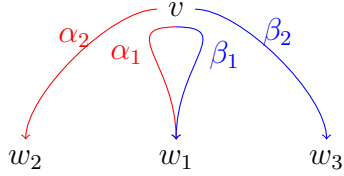
\begin{figure}[htb]
\begin{tikzpicture}[scale=2]
\node (v) at (1,1)  {$v$};
 \node (w_1) at (1,0) {$w_1$};
 \node (w_2) at (0,0) {$w_2$};
 \node (w_3) at (2,0) {$w_3$};
 \draw[->,red]  (v.west) ..  node[above]{$\alpha_2$} controls+(left:3mm) and +(up:3mm) ..
 (w_2.north) ;
 \draw[->,red] (v.south) .. node[below, left]{$\alpha_1$}  controls+(left:4mm) and +(up:5mm) ..
 (w_1.north);
\draw[->,blue] (v.south) .. node[below, right]{$\beta_1$}
controls+(right:4mm) and +(up:5mm) ..
 (w_1.north);
\draw[->,blue] (v.east) .. node[above]{$\beta_2$}
controls+(right:3mm) and +(up:3mm) ..
 (w_3.north);
\end{tikzpicture}
\caption{The separated graph $(E,C)$}
\end{figure}}
\end{center}

\begin{center}{
\begin{figure}[htb]
\begin{tikzpicture}[scale=2]
\node (v) at (0,1)  {$v$};
 \node (w_1) at (0,0) {$w_1$};
 \node (w_2) at (0,-1) {$w_2$};
 \draw[->,red]  (v.south) ..  node[below, left]{$\alpha_1$} controls+(left:4mm) and +(up:5mm) ..
 (w_1.north) ;
 \draw[->,red] (v.west) .. node[below, left]{$\alpha_2$}  controls+(left:14mm) and +(up:14mm) ..
 (w_2.north);
\draw[->,blue] (v.east) .. node[below, right]{$\beta_2$}
controls+(right:14mm) and +(up:14mm) ..
 (w_2.north);
\draw[->,blue] (v.south) .. node[below, right]{$\beta_1$}
controls+(right:4mm) and +(up:5mm) ..
 (w_1.north);
\end{tikzpicture}
\caption{The separated graph $(F,D)$}
\end{figure}
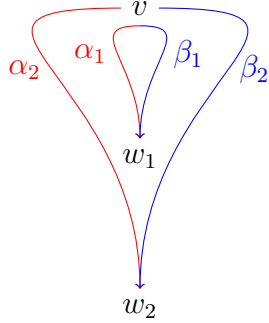}
\end{center}

We start by analyzing the full C*-algebras.

\begin{lemma}
\label{lem:full-algebras}
\begin{enumerate}
\item
 Let $(E,C)$ be the separated graph associated to the presentation
 $\langle a,b,c\mid a+b= a+c \rangle $ (see Figure 1). Then
the corner $vC^* (E,C)v$ is the universal unital C*-algebra
generated by a partial isometry.
\item Let $(F,D)$ be the separated graph associated to the
presentation $\langle a,b \mid a+b=a+b \rangle$ (see Figure 2). Then
$$vC^* (F,D) v \cong C^*( (\bgast _{\Z}\Z_2)\rtimes \Z) \,$$
where $\Z$ acts on $\bgast_{\Z} \Z_2$ by shifting the factors of the
free product. Moreover $vC^*(E,C)v$ is a C*-subalgebra of
$vC^*(F,D)v$.
\end{enumerate}
\end{lemma}

\begin{proof}
(1) We have $r(\alpha _1)=w_1=r(\beta_1)$ and $r(\alpha _2)=w_2$,
$r(\beta_2)=w_3$. Let $s=\beta_1\alpha _1^*\in vC^*(E,C)v$. Then
$s^*s= \alpha_1 \alpha_1 ^*=:p$ and $ss^*=\beta_1\beta_1^*=: q$. Let
$\mathfrak G$ be the universal unital C*-algebra generated by a
partial isometry $w$. Then there is a unique unital $*$-homomorphism
$ \mathfrak G\to vC^*(E,C)v$ sending $w$ to $s$. It is not
difficult, using the universal property of $C^*(E,C)$, to build an
inverse of this homomorphism.

\medskip

(2) Here we have $r(\alpha _i)=r(\beta_i)=w_i$ for $i=1,2$. Consider
the unitary $u=\beta _1\alpha_1^*+\beta _2\alpha_2^*$ in
$vC^*(F,D)v$, and the projection $p_0:=\alpha _1\alpha _1^*$. Let
$\mathfrak U= C^*(u',p')$ be the universal C*-algebra generated by a
unitary $u'$ and a projection $p'$. There exists a unique
$*$-homomorphism $\mathfrak U\to vC^*(F,D)v$ sending $u'$ to $u$ and
$p'$ to $p$. It is easily seen (again using the universal property
of $C^*(F,D)$) that this map is an isomorphism. So we will identify
$u'$ with $u$ and $p'$ with $p$, and we shall write $\mathfrak U
=vC^*(F,D)v$.

\medskip

We will now show that $\mathfrak U\cong C^* ( (\bgast
_{\Z}\Z_2)\rtimes \Z)$. Write $p_0=p$ and $p_n=u^np(u^*)^n$. Let
$(u_n)_{n\in \Z}$ denote the unitaries in $C^* ( (\bgast
_{\Z}\Z_2)\rtimes \Z) $ corresponding to the generators of the
different copies of $\Z_2$, and let $z$ be the unitary implementing
the action of $\Z$ on $\bgast _{\Z} \Z_2$. We may define a
$*$-homomorphism $\varphi \colon \mathfrak U\to C^* ( (\bgast
_{\Z}\Z_2)\rtimes \Z)$ by sending $u$ to $z$ and $p$ to
$\frac{1-u_0}{2}$. On the other hand, we have a unitary
representation of $(\bgast _{\Z}\Z_2)\rtimes \Z$ on $\mathfrak U$
obtained by sending $u_n$ to $1-2p_n$ and $z$ to $u$. It is
straightforward to check that $\varphi$ and $\psi$ are mutually
inverse.

\medskip

Now we show that the canonical homomorphism $\eta \colon \mathfrak G
\to \mathfrak U$ sending $s$ to $up$ is an isometry. There is a
faithful representation $\rho $ of $ \mathfrak G$ on a separable
Hilbert space $H$ such that both $(1 -\rho (p))H$ and $(1-\rho
(q))H$ are infinite-dimensional. Therefore there is a unitary $U$ on
$H$ extending $\rho (s)$, so that $U\rho (p)=\rho (s)$. It follows
that the  $*$-homomorphism $\rho \colon \mathfrak G \to B(H)$
factors through $\mathfrak U$, that is there is a $*$-homomorphism
$\varphi \colon \mathfrak U\to B(H)$ such that $\rho = \varphi \circ
\eta$. Since $\rho $ is faithful, we see that $\eta$ is injective,
and so it is an isometry.
\end{proof}

The universal {\it non-unital} C*-algebra $\mathfrak G '$ generated
by an isometry has recently been studied by Brenken and Niu in
\cite{BN}. The C*-algebra $\mathfrak G = vC^*(E,C)v$ of Lemma
\ref{lem:full-algebras}(1) is just the unitization of $\mathfrak G
'$. Some properties of $C^*(E,C)$ can thus be derived from
\cite{BN}, for instance we see from \cite[Corollary 1]{BN} that
$C^*(E,C)$ is a non-exact C*-algebra.

\medskip

The following result was obtained in collaboration with Ken
Goodearl.

\begin{proposition}
\label{prop:reducedgpCstar} Let $(F,D)$ be the separated graph
associated to the presentation $\langle a,b \mid a+b=a+b \rangle$.
Then
$$v\Cstred (F,D)v\cong C^*_r((\bgast _{\Z} \Z_2)\rtimes \Z)
\cong (C^*_r(\bgast_{\Z} \Z_2))\rtimes _r \Z .$$
\end{proposition}

\begin{proof}
We follow with the notation introduced in the proof of Lemma
\ref{lem:full-algebras}. We have
$$(\gotA ,\Phi )= (\Cstred (F,D),\Phi)= (A, \Phi_A)*(B, \Phi _B)\, ,$$
where $A$ and $B$ are the usual graph C*-algebras of the graphs
corresponding to the edges $\{\alpha _1,\alpha _2\}$ and
$\{\beta_1,\beta_2\}$ respectively, and the maps $\Phi_A$ and
$\Phi_B$ are the canonical conditional expectations onto $\C v+\C
w_1+\C w_2$, as defined in \cite[Theorem 2.1]{AG2}.

\medskip

Put $T=M_2(C^*_r((\bgast _{\Z} \Z_2)\rtimes \Z))$, and denote by
$e_{ij}$, $1\le i,j\le 2$ the canonical matrix units in $T$. Set
$$p_i^{\pm} =\frac{1\mp u_i}{2},\qquad i\in \Z\, .
$$
 We
define unital $*$-homomorphisms $\sigma _A\colon A\to T$ and $\sigma
_B\colon B\to T$ by
$$\sigma_A(v)=\sigma_B (v)= e_{11},\quad \sigma_A(w_1)= \sigma _B(w_1)= p_0^+e_{22},\quad
\sigma_A(w_2)=\sigma _B (w_2)= p_0^- e_{22} ,$$
$$\sigma _A(\alpha _1)= p_0^+e_{12},\quad \sigma_A(\alpha_2)= p_0^-e_{12},
\quad \sigma_B( \beta_1)= zp_0^+e_{12},\quad \sigma_B(\beta_2)=
zp_0^- e_{12} .$$

\medskip

Let $\Theta \colon T\to \C e_{11}+ \C p_0^+ +\C p_0^-$ be the
conditional expectation given by
$$\Theta \Big(\begin{pmatrix}
 a_{11} & a_{12} \\ a_{21} & a_{22}\end{pmatrix}\Big) =
\tau (a_{11}) e_{11} + \tau (p_0^+ a_{22} p_0^+)p_0^+ + \tau (p_0^-
a_{22}p_0^-)p_0^- \, ,$$ where $\tau$ is the canonical faithful
trace on the reduced group C*-algebra $C^*_r ((\bgast _{\Z}
\Z_2)\rtimes \Z)$. In order to check that $(\gotA, \Phi)\cong (T,
\Theta)$, it suffices to show that the conditions (1)--(5) in
Definition \ref{def:redamprod} are satisfied. All conditions are
easily verified, with the exception of (4).

\medskip

To show (4), we compute the kernels of $\Theta|_{\sigma_A(A)}$ and
$\Theta|_{\sigma_B(B)}$ to be $Z_0= \C u_0e_{11}
+A_0e_{12}+A_0e_{21}$ and $Z_1=\C u_1e_{11}+A_1ze_{12}+
A_0z^*e_{21}$ respectively, where $A_i=\C 1+\C u_i$ is the
subalgebra of $C^*_r((\bgast _{\Z} \Z_2)\rtimes \Z)$ generated by
$u_i$. Then we have to show that $\Theta (\alpha)= 0$ for every
$\alpha \in \Lambda^{{\rm o}}(Z_0,Z_1)$. This is shown by induction
on the length of $\alpha $. In order to prove it, we introduce the
following notation. For $i\in \Z$, denote by $F_i$ the linear span
of the set of reduced words in $\{u_n:n\in \Z \}$ ending in $u_i$.
We will write $F_{\le i}= \sum _{j\le i} F_j$ and
$\widetilde{F}_{\le i}=\C 1+F_{\le i}$, and similarly for $F_{\ge
i}$ and $\widetilde{F}_{\ge i}$.

\medskip

One shows by induction on the length of $\alpha \in \Lambda ^{{\rm
o}}(Z_0,Z_1)$ that if $\alpha $ ends in $Z_0$, then
$$\alpha \in L_{11}^{(0)}e_{11}+
L_{12}^{(0)}e_{12}+L_{21}^{(0)}e_{21}+L_{22}^{(0)}e_{22} ,$$ where
$$L_{11}^{(0)}=F_{\le 0}+ \sum _{i=1}^n[\widetilde{F}_{\le i}]z^i+
\sum_{j=1}^n [F_{\le -j}] (z^*)^{j} ,$$
$$L_{22}^{(0)}=[F_{\ge 1}]A_0+ \sum _{i=1}^{n-1}[F_{\ge i+1}]A_iz^i+
 \sum_{j=1}^{n} [\widetilde{F}_{\ge -j+1}]A_{-j} (z^*)^{j} ,$$
and $L_{21}^{(0)}=A_0+L^{(0)}_{11}$, $L_{12}^{(0)}=
A_0+L_{22}^{(0)}$, for suitable $n\ge 0$.

Correspondingly, if $\alpha$ ends in $Z_1$, then
$$\alpha \in L_{11}^{(1)}e_{11}+
L_{12}^{(1)}e_{12}+L_{21}^{(1)}e_{21}+L_{22}^{(1)}e_{22} ,$$ where
$$L_{11}^{(1)}= L_{21}^{(1)}= F_{\ge 1} + \sum _{i=1}^n [F_{\ge i+1}] z^i
+ \sum _{j=1}^{n+1} [\widetilde{F}_{\ge -j+1}] (z^*)^{j} , $$ and
$$L_{22}^{(1)}=L_{12}^{(1)}= [F_{\le -1} ]A_0
 + \sum _{i=1}^{n+1} [\widetilde{F}_{\le i-1}] A_iz^i + \sum
_{j=1}^{n-1} [F_{\le -j-1}] A_{-j}(z^*)^j .$$

This concludes the proof.
\end{proof}

\begin{corollary}
 Let $(F,D)$ be the separated graph
associated to the presentation $\langle a,b \mid a+b=a+b \rangle$.
Then $\Cstred (F,D)$ is a simple C*-algebra.
\end{corollary}

\begin{proof}
By Proposition \ref{prop:reducedgpCstar}, it is enough to show that
$C_r ((\bgast _{\Z} \Z_2)\rtimes \Z)$ is simple. First observe that
$$(C^*_r(\bgast _{\Z} \Z_2),\tau) = \bgast _{i\in \Z} (C^*_r
(\Z_2),\tau_i),$$ that is $C_r^*(\bgast _{\Z} \Z_2)$ is the reduced
crossed product of countably many  group C*-algebras $C^*_r(\Z_2)$,
with respect to their canonical tracial states $\tau_i$, and so it
is simple by an application of Avitzour's Theorem. By Kishimoto's
Theorem \cite[Theorem 3.1]{Kishi}, in order to show that $C^*_r
(((\bgast _{\Z} \Z_2)\rtimes \Z))\cong (C_r^*(\bgast _{\Z}
\Z_2))\rtimes _r\Z$ is simple, it is enough to show that the action
of each non-trivial element of $\Z$ on $C_r^*(\bgast _{\Z} \Z_2)$ is
outer. To show this, we first show that the relative commutant of
$\gotB := C^*_r (\bgast _{\Z} \Z_2)$ in $\mathfrak C := C_r ((\bgast
_{\Z} \Z_2)\rtimes \Z)$ is trivial. For this, we use an argument
similar to the one in \cite[proof of Claim 3]{ChoDyk}. We put $G:=
(\bgast _{\Z} \Z_2)\rtimes \Z$.

\medskip

Suppose that $x\in \mathfrak C$ and $x$ commutes with $\gotB$. We
will show that $x_0= x-\tau (x)1$ is zero. Suppose, to obtain a
contradiction, that $x_0\ne 0$. Since $\tau$ is faithful, $\|
x_0\|_2=\tau (x_0^*x_0)^{1/2}>0$. Choose $\epsilon >0$ so that
$0<\epsilon < \frac{ \|x_0\|_2}{3}$. There is an element $y$ in the
group algebra $\C G$ such that
$$y=\sum_{g\in F} \lambda _g g \, ,$$
where $F$ is a finite subset of $G\setminus \{1\}$, $\lambda _g\in
\C\setminus \{0\}$, and $\| x_0-y\| <\epsilon $. We consider the
canonical expression of elements in $G$, as $wz^j$, where $w$ is a
reduced word in the $u_i$'s and $j$ is an integer. Let $I$ be the
finite subset of $\Z$ consisting of those integers $n$ such that
$u_n$ is involved in the canonical expression of some of the
elements of $F$. Let $J$ be the (finite) set of integers that appear
as powers of $z$ in the canonical expression of the elements of $F$.
Take $N\in \N$ big enough so that $N > \text{max}\{i, i'-j\mid
i,i'\in I, j\in J \}$,  and write $v= u_N$.

\medskip

We claim that $vy$ and $yv$ are orthogonal with respect to the inner
product on $\mathfrak C$ induced by $\tau$, that is, $\tau
(v^*y^*vy)=0$. Indeed we have
$$\tau (v^*y^*vy)=\tau ( \sum _{g,h\in F} \ol{\lambda} _g\lambda _h u_N
g^{-1} u_N h ) =0\, ,
$$
because of the choice of $N$. Now, by orthogonality of $vy$ and
$yv$, we have $\|vy-yv\|\ge \| vy-yv \|_2> \| vy\|_2 =\| y\|_2$, and
thus
$$\| vx_0 -x_0x\| \ge \| vy-yv\| -2\epsilon > \| y \|_2 -2\epsilon \ge \| x_0\|_2 -3\epsilon
>0, $$
contradicting the fact that $x_0$ centralizes $\gotB$.

\medskip

It is now easy to see that, for every non-zero integer $m$, the
action $\alpha ^m$ is outer on $\gotB$. Indeed, if $\alpha ^m$ is an
inner automorphism of $\gotB$, induced by a unitary $d$ in $\gotB$,
then there is $\lambda \in \mathbb T$ such that $d= \lambda z^m$ in
$\mathfrak C=\gotB \rtimes _r \Z$, which is a contradiction.
Finally, Kishimoto's Theorem \cite[Theorem 3.1]{Kishi} gives that
$\mathfrak C$, and so $\Cstred (F,D)$, is a simple C*-algebra.
\end{proof}

Finally, we show that the embedding of $vC^*(E,C)v$ into
$vC^*(F,D)v$ established in Lemma \ref{lem:full-algebras}(2) extends
to the reduced setting.

\begin{proposition}
\label{prop:embedding-inredsettin} Adopt the notation of Lemma
\ref{lem:full-algebras}. Then there is a trace-preserving embedding
$v\Cstred (E,C)v\hookrightarrow v\Cstred (F,D)v.$
\end{proposition}

\begin{proof}
Consider the subalgebra $\mathfrak D = (1\oplus w_2 )M_2(\Cstred
(F,D))(1\oplus w_2)$ of $M_2(\Cstred (F,D))$. Set $C=\C v \oplus \C
w_1\oplus \C w_2 \oplus \C w_3$. Define $º\widetilde{\Phi }\colon
\mathfrak D \to C$ by
$$\widetilde{\Phi}\Big( \begin{pmatrix} x_{11} & x_{12} \\ x_{21} &
x_{22}
\end{pmatrix}
\Big)  = (\Phi (x_{11}), \tau _{w_2}(x_{22}) w_3 )\, ,$$ where $\Phi
\colon \Cstred (F,D)\to \C v\oplus \C w_1\oplus \C w_2$ is the
canonical conditional expectation and $\tau_{w_2}$ is the state of
$w_2\Cstred (F,D) w_2$ given by $\Phi (x)= \tau _{w_2}(x) w_2$ for
$x\in w_2\Cstred (F,D) w_2$. We look $C$ as a C*-subalgebra of
$\mathfrak D$ by sending $v,w_1,w_2$ to the corresponding vertices
in $(1\oplus 0) M_2(\Cstred (F,D))(1\oplus 0)$ and sending $w_3$ to
$0\oplus w_2$. With this embedding in mind, $\widetilde{\Phi}$ is a
faithful conditional expectation from $\mathfrak D$ onto $C$.

Write $\Cstred (E,C)=A_1*_C A_2$, where $A_1 = C^*(C, \alpha
_1,\alpha_2 )$ and $A_2=C^*(C, \beta_1,\beta_2 )$. We define
$*$-homomorphisms $\sigma _i\colon A_i\to \mathfrak D$ by sending
canonically $C$ to $\mathfrak D$ as above, and putting
$$\sigma_1 (\alpha _1) = \alpha_1\oplus 0,\quad \sigma (\alpha
_2)=\alpha_2\oplus 0,\quad \sigma_2(\beta_1)= \beta_1\oplus 0, \quad
\sigma_2( \beta_2) = \begin{pmatrix} 0 & \beta_2 \\ 0 & 0
\end{pmatrix}. $$
In order to show that these maps define an isomorphism from $\Cstred
(E,C)$ onto $C^*(\sigma_1 (A_1)\cup \sigma_2 (A_2) )$, it suffices
to check conditions (1)-(5) of Definition \ref{def:redamprod}. All
properties are obvious with the exception of (4). In order to check
(4), set
$$Z_1=\{
1-2\alpha_1\alpha_1^*,\alpha_1,\alpha_1^*,\alpha_2,\alpha_2^*\},
\quad Z_2 =\{ 1-2\beta_1\beta_1^*,
\beta_1,\beta_1^*,\beta_2,\beta_2^*\},$$ and let $a_1a_2\cdots
a_n\in \Lambda^{{\rm o}}(Z_1,Z_2)$, where $a_i\in Z_{\iota_i}$, with
$\iota_i\ne \iota_{i+1}$ for $i=1,\cdots , n-1$. Then we have to
prove that $\widetilde{\Phi} (\sigma
_{\iota_1}(a_1)\sigma_{\iota_2}(a_2) \cdots
\sigma_{\iota_n}(a_n))=0$. This is obvious if all the letters in
$a_1a_2\cdots a_n$ are different from $\beta_2$ and $\beta_2^*$. The
nonzero expressions involving $\beta_2$ or $\beta_2^*$ give terms in
$\mathfrak D$ which have one of the following forms:
$$\begin{pmatrix} 0 & a\beta_2 \\
0 & 0
\end{pmatrix},\quad \begin{pmatrix} 0 & 0 \\ \beta_2^*a & 0
\end{pmatrix},\quad  \text{or} \quad  \begin{pmatrix} 0 & 0\\ 0 &  \beta_2^* a \beta_2
\end{pmatrix}, $$
with $\Phi (a\beta_2) =0$,  $\Phi (\beta_2^*a)=0$,  and $\Phi (\beta
_2^* a\beta ) =0$ in the respective cases. Consequently, $
\widetilde{\Phi} (\sigma _{\iota_1}(a_1)\sigma_{\iota_2}(a_2) \cdots
\sigma_{\iota_n}(a_n))=0$ in all cases, as desired.
\end{proof}

It remains an open problem to determine whether the C*-algebra
$\Cstred (E,C)$ considered in the above result is simple.

\section{Some remarks on K-theory}
\label{sect:k-theory}

We recall from \cite{AG2} the following conjecture:

\begin{conjecture}\cite[7.6]{AG2}
\label{op-vmonoid} Let $(E,C)$ be a finitely separated graph. Let
$M(E,C)$ be the abelian monoid with generators $\{a_v\mid v\in
E^0\}$ and relations given by $a_v=\sum _{e\in X} a_{r(e)}$ for all
$v\in E^0$ and all $X\in C_v$. Then the natural map $M(E,C)\to
\mon{C^*(E,C)}$ is an isomorphism.
\end{conjecture}

Let $L(E,C)$ be the dense $*$-subalgebra of $C^*(E,C)$ generated by
the canonical generators of $C^*(E,C)$.  It was shown in
\cite[Theorem 4.3]{AG} that there is a natural isomorphism
$M(E,C)\to \mon{L(E,C)}$, sending $a_v$ to $[v]\in \mon{L(E,C)}$,
where $\mon{L(E,C)}$ is the abelian monoid of isomorphism classes of
finitely generated projective right modules over $L(E,C)$. So the
above conjecture is equivalent to the question of whether the
natural induced map $\mon{L(E,C)}\to \mon{C^*(E,C)}$ is an
isomorphism. The answer is positive in the non-separated case
\cite[Theorem 7.1]{AMP}.

We now make two weaker conjectures, which can be checked in various
situations of interest.

\begin{conjecture}
\label{weaker-conj2} Let $(E,C)$ be a finitely separated graph. Then
the natural map
$$M(E,C)  \to \mon{C^*(E,C)}$$
is injective.
\end{conjecture}

\begin{conjecture}
\label{weaker-conj1}
 Let $(E,C)$ be a finitely separated graph. Let
$(G(E,C), G(E,C)^+)$ be the Grothendieck group of $M(E,C)$, with the
canonical pre-ordered structure given by taking $G(E,C)^+=\iota
(M(E,C))$, where $\iota \colon M(E,C)\to G(E,C)$ is the canonical
map (note that $\iota $ does not need to be injective). Then we have
a natural homomorphism
$$(G(E,C), G(E,C)^+)\to (K_0(C^*(E,C)), K_0(C^*(E,C))^+) $$
of partially pre-ordered abelian groups.
 We conjecture that this map is an isomorphism. By \cite[Theorem 5.2]{AG2} the map
$G(E,C)\to K_0(C^*(E,C))$ is an isomorphism, so the conjecture is
 that the order structure in $K_0(C^*(E,C))$ is the one given by $G(E,C)^+$.
\end{conjecture}

Let us show that Conjecture \ref{op-vmonoid} implies a positive
answer to a question of R\o rdam and Villadsen \cite[Question
2.1(a)]{RV}. Recall that an ordered group is a pair $(G,G^+)$ where
$G$ is an abelian group, $G^+\subseteq G$, and
$$ G^+ + G^+\subseteq G^+, \qquad G^+-G^+=G, \qquad G^+\cap
-G^+=\{0\} .$$

\begin{remark}
\label{rem:questionRV} If Conjecture \ref{op-vmonoid} holds, then
for every ordered abelian group $(G,G^+)$ there is a stably finite
C*-algebra $\mathcal A$ such that $(K_0(\mathcal A), K_0(\mathcal
A)^+)\cong (G,G^+)$.
\end{remark}

Indeed, since $G^+$ is a conical abelian monoid, we may take a
presentation $\langle \mathcal X\mid \mathcal R \rangle$ of $G^+$ as
indicated in the proof of \cite[Proposition 4.4]{AG}. Then, with
$(E,C)$ being the separated graph associated to this presentation,
we have $M(E,C)\cong G^+$ (\cite[4.4]{AG}). Since the Grothendieck
group of $G^+$ is $G$ we obtain from Conjecture \ref{op-vmonoid}
that $(G,G^+)\cong (K_0(\mathcal A), K_0(\mathcal A)^+)$, where
$\mathcal A= C^*(E,C)$. It remains to verify that $\mathcal A$ is
stably finite. But since $\mon{\mathcal A}\cong G^+$ embeds into a
group, it is clear that all projections in $M_{\infty}(\mathcal A)$
are finite, as desired.

\medskip

However the validity of Conjecture \ref{op-vmonoid} seems to be
known in very few cases in the non-separated case.

\medskip

The validity of the weaker conjectures \ref{weaker-conj2},
\ref{weaker-conj1} can be checked in several cases. For instance, it
holds for the separated graph $(E,C)$ with just one vertex and the
sets in the partition $C$ reduced to singletons, because then
$C^*(E,C)$ is just the full group C*-algebra $C^*(\mathbb F _n)$,
where $n$ is the number of edges, and for the C*-algebras
$C^*(E(n,n), C(n,n))$ (see the proof of \cite[Theorem 6.3]{AG2}).
For the C*-algebras $\mathcal A =M_k(\C)*M_l(\C)$, with
$\text{gcd}(k,l)=1$ and at least one of $k$ or $l$ prime, it follows
from \cite[Theorem 3.6]{RV} that $(K_0(\mathcal A), K_0(\mathcal
A)^+)\cong (\Z, \langle k,l\rangle )$, with the integers $k,l$
represented in $K_0(\mathcal A)$ by the classes of the minimal
projections in $M_l(\C)$ and $M_k(\C)$ respectively. The algebra
$\mathcal A$ is Morita-equivalent to the C*-algebra of the separated
graph associated to the presentation $\langle a,b\mid ka=lb\rangle$,
and so we obtain for these particular values of $k,l$ that both
Conjectures \ref{weaker-conj2} and \ref{weaker-conj1} are true.

\medskip

As another interesting example, we consider the C*-algebra $\mathcal
A=\mathcal O_n *\mathcal O _m= C^*(E,C)$ treated in Proposition
\ref{prop:one-vertex}. In this case $K_0(\mathcal A)=K_0(\mathcal
A)^+\cong \Z_{d}$, where $d=\text{gcd}(n-1,m-1)$ (by \cite[Theorem
5.2]{AG2}), and
$$M(E,C)=\langle a\mid a=na=ma\rangle = \langle a \mid a=
(d+1)a\rangle $$ so both Conjectures  \ref{weaker-conj2} and
\ref{weaker-conj1} hold. Note that $\Cstred(E,C)$ is purely infinite
simple by \ref{prop:one-vertex}, and that $K_0(\mathcal A)\cong
K_0(\Cstred (E,C))$ by Germain's Theorem (\cite{Germain}), so we
obtain the very precise information $\mon{\Cstred (E,C)}=\langle a
\mid a=(d+1)a\rangle $ for the reduced graph C*-algebra.

\medskip

In contrast, the natural map $M(E,C)\to \mon{\Cstred (E,C)}$,
defined by composing the natural map $M(E,C)\to \mon{C^*(E,C)}$ with
the homomorphism induced by the natural surjection $C^*(E,C)\to
\Cstred (E,C)$, fails to be injective or surjective in many cases.
Using our main result, we will show that it may even happen that
$M(E,C)$ is stably finite (i.e. $x+y=x\implies y=0 \,\,  \forall
x,y\in M(E,C)$) but $\Cstred (E,C)$ is purely infinite simple. To
see this, we need a monoid-theoretic lemma.

\begin{lemma}
\label{lem:finiteM} Let $F$ be the free abelian monoid on free
generators $a_1,a_2,\dots ,a_n$. Let
$$x=\sum _{i=1}^n
r_ia_i,\qquad y =\sum_{i=1}^n s_i a_i $$ be nonzero elements in $F$.
Let $M$ be the conical abelian monoid ${F/\sim}$ where $\sim$ is the
congruence on $F$ generated by $(x,y)$. Then $M$ contains infinite
elements if and only if either $x<y$ or $y<x$ in the usual order of
$F$.
\end{lemma}

\begin{proof}
If $x<y$ or $y<x$ then clearly $[x]$ and $[y]$ are infinite elements
in $M$.

Observe that $\sim $ agrees with the congruence on $F$ defined as
follows: For $a,b\in F$, set $a\backsimeq b$ in  case there exists a
sequence $z_0,z_1,\dots ,z_r$, $r\ge 0$,  of elements of $F$ such
that $a=z_0$, $b=z_r$ and for each $i=0,\dots ,r-1$, either
$z_i=y_i+x$ and $z_{i+1}= y_i+y$ for some $y_i\in F$, or $z_i=y_i+y$
and $z_{i+1}= y_i+x$ for some $y_i\in F$. Hence, if $a\sim b$ then
there is $t\in \Z$ such that $b=a+t(y-x)$ in the free abelian group
$\Z ^n$. Thus if $a\sim a+c$ with $c\in F\setminus \{0\}$, then
there is $t\in \Z$ such that $c=t(y-x)$, and this implies either $x
< y$ or $y<x$.
\end{proof}

It is interesting to observe that the conditions giving that
$M(E,C)$ is stably finite are the same as the ones giving that
$C^*(E,C)$ is stably finite, as the following proposition and the
above lemma show. This provides further evidence to the validity of
Conjecture \ref{op-vmonoid}. Recall that a C*-algebra $A$ is termed
{\it residually finite dimensional} if it admits a separating family
of finite-dimensional $*$-representations.

\begin{proposition}
\label{prop:RFD} Let $(E,C)$ be the separated graph associated to
the presentation $\langle a_1,\dots ,a_n\mid \sum _{i=1}^n
r_ia_i=\sum _{i=1}^n s_ia_i\rangle $. Consider the nonzero vectors
${\bf r}= (r_1,\dots ,r_n)$ and ${\bf s}=(s_1,\dots ,s_n)$ in
$\Z^n$. Then the following conditions are equivalent:
\begin{enumerate}
\item[(i)] $C^*(E,C)$ is residually finite dimensional.
\item[(ii)] $C^*(E,C)$ admits a faithful tracial state.
\item[(iii)] $C^*(E,C)$ is stably finite.
\item[(iv)] $C^*(E,C)$ is finite.
\item[(v)] ${\bf r}\nless {\bf s}$ and ${\bf s}\nless {\bf r}$ in the usual
order of $\Z^n$.
\end{enumerate}
\end{proposition}

\begin{proof}
The implications $(i)\implies (ii) \implies (iii) \implies (iv)$ are
general, well-known facts.

$(iv)\implies (v)$.  If ${\bf r} < {\bf s}$ or ${\bf s} <{\bf r}$ in
the usual order of $\Z^n$ then one can easily see that the
projection
$$v\sim \bigoplus _{i=1}^n r_i\cdot w_i \sim \bigoplus
_{i=1}^n s_i \cdot w_i $$ is infinite in $C^*(E,C)$.

$(v)\implies (i)$.  Assume that ${\bf r}\nless {\bf s}$ and ${\bf
s}\nless {\bf r}$ in the usual order of $\Z^n$. Then we shall show
that $C^*(E,C)$ is residually finite dimensional by applying
\cite[Theorem 4.2]{ADEL}, which asserts that a full amalgamated free
product $A*_CB$ of finite dimensional C*-algebras is residually
finite dimensional if and only if there are faithful tracial states
$\tau_A$ on $A$ and $\tau_B$ on $B$ whose restrictions to $C$ agree.

Since ${\bf r}\nless {\bf s}$ and ${\bf s}\nless {\bf r}$, either
${\bf r}={\bf s}$ or there are indices $i,j$ such that $r_i<s_i$ and
$s_j<r_j$. In the former case, it follows from Proposition
\ref{prop:balancedcase} that there is a faithful state $\tau$ on
$C=\C^{n+1}$ such that $\tau_A=\tau \circ \Phi_A$ and $\tau_B=
\tau\circ \Phi _B$ are faithful tracial states on $A$ and $B$
respectively. So \cite[Theorem 4.2]{ADEL} gives the result. In the
latter case, without loss of generality, we shall assume that
$r_1<s_1$ and $s_2<r_2$.

 We will use the concrete representations of $A=C^*(E_X)$ and
$B=C^*(E_Y)$ as finite products of matrix algebras introduced in
Section \ref{sect:pur-infinite}, so that $A= \prod _{i=1}^n
M_{s_i+1}(\C)$, $B=\prod_{i=1}^n M_{r_i+1}(\C)$ and $C=\C^{n+1}$,
and the embeddings $\iota _A\colon C\to A$ and $\iota _B\colon C\to
B$ are as specified in Section \ref{sect:pur-infinite}. We want to
define faithful tracial states $\tau _A= \sum _{i=1}^n \gamma _i
\text{Tr}_{s_i+1}$ on $A$ and $\tau _B=\sum _{i=1}^n \delta _i
\text{Tr}_{r_i+1}$ on $B$, with $\gamma_i >0$ and $\delta _i>0$ for
all $i$, and $\sum _{i=1}^n \gamma_i =\sum _{i=1}^n \delta _i =1$,
such that the restrictions of $\tau _A$ and $\tau _B$ to
$C=\C^{n+1}$ agree. This is equivalent to the identities:
\begin{equation}
\label{eq:gamm-deltas} \delta _i = \Big(\frac{r_i+1}{s_i+1}\Big)
\gamma_i,\qquad i=1,\dots ,n.
\end{equation}
For $i=2,\dots ,n$, put
$$\Delta _i = \Big(\frac{r_i+1}{s_i+1}\Big)-
\Big( \frac{r_1+1}{s_1+1}\Big)\, .$$
Set $\Gamma =
(r_2+1)(s_1+1)-(s_2+1)(r_1+1)$, and observe that $\Gamma >0$ by our
hypothesis that $r_1<s_1$ and $s_2<r_2$. Set
\begin{equation}
\label{eq:gamma'} \gamma _1'= \frac{(s_1+1)(r_2-s_2)}{\Gamma
},\qquad \gamma _2'= \frac{(s_2+1)(s_1-r_1)}{\Gamma } \, ,
\end{equation}
and observe that $\gamma _1'>0$, $\gamma _2' >0$ and $\gamma
_1'+\gamma _2'=1$. Now define $\gamma _1, \gamma _2$ by
\begin{equation}
\label{eq:gamma12} \gamma_1 = \gamma _1'+\sum _{i=3}^n \Big(
\frac{\Delta_i}{\Delta_2}-1\Big) \gamma_i,\qquad \gamma _2= \gamma
_2'-\sum _{i=3}^n \Big( \frac{\Delta_i}{\Delta_2}\Big) \gamma _i \,
,\end{equation} where $\gamma _3,\dots, \gamma _n$ are chosen to be
positive numbers small enough to make $\gamma _1$ and $\gamma _2$
both positive. With this choice of $\gamma_1,\dots ,\gamma _n$ one
has that $\gamma_i >0$ for all $i$ and that $\sum _{i=1}^n \gamma _i
=1$. Now, the equations (\ref{eq:gamm-deltas}) define positive
numbers $\delta _i$, and it is easily checked that $\sum _{i=1}^n
\delta_i =1$.

Hence, the formulas $\tau _A= \sum _{i=1}^n \gamma _i
\text{Tr}_{s_i+1}$ and $\tau _B=\sum _{i=1}^n \delta _i
\text{Tr}_{r_i+1}$ define faithful tracial states on $A$ and $B$
respectively, such that the restrictions of $\tau _A$ and $\tau _B$
to $C=\C^{n+1}$ agree. It follows from \cite[Theorem 4.2]{ADEL} that
$C^*(E,C)= A*_CB$ is residually finite dimensional.
\end{proof}

We observe that, in view of Lemma \ref{lem:full-algebras} the above
result incorporates \cite[Theorem 2.2]{BN}.

\begin{example}
\label{examplM(E,C)st} There exists separated graphs $(E,C)$ such
that $M(E,C)$ is a stably finite monoid and $C^*(E,C)$ is a stably
finite C*-algebra, but $\Cstred (E,C)$ is purely infinite simple,
and moreover the natural map $M(E,C)\to \mon{\Cstred (E,C)}$ is not
injective.
\end{example}

\begin{proof}
Take for instance the separated graph associated to the one-relator
monoid $\langle a,b \mid 3a+2b=2a+4b \rangle$. By Lemma
\ref{lem:finiteM}, $M(E,C)$ is a stably finite monoid and, by
Proposition \ref{prop:RFD}, $C^*(E,C)$ is a stably finite
C*-algebra. However, by Theorem \ref{thm:motivating}, we have that
$\Cstred (E,C)$ is purely infinite simple. In particular we obtain
that $\mon{\Cstred (E,C)}\setminus \{0\}=K_0(\Cstred (E,C))$ is
cancellative. Thus $a\ne 2b$ in $M(E,C)$ but $a=2b$ in $\mon{\Cstred
(E,C)}$. This shows the result.
\end{proof}

\section*{Acknowledgements} It is a pleasure to thank Ken Goodearl for his
contribution to this work. I am grateful to Ruy Exel and Terry
Loring for their useful comments.


\begin{thebibliography}{12}

\bibitem{AA} G. Abrams, G. Aranda Pino, \emph{ The Leavitt path algebra of a
graph}, J. Algebra  {\bf 293}  (2005),  319--334.

\bibitem{ADR} P. Ara, K. Dykema, M. R\o rdam, \emph{Correction of proofs in
``Purely infinite simple C*-algebras arising from free product
constructions'' and a subsequent paper}, Canad. J. Math., to appear.

\bibitem{AG} P. Ara, K. R. Goodearl, \emph{Leavitt path algebras of
separated graphs}. J. reine angew. Math. (2011).

\bibitem{AG2} P. Ara, K. R. Goodearl, \emph{C*-algebras of
separated graphs}, J. Funct. Anal. {\bf 261} (2011), 2540--2568.

\bibitem{AMP} P. Ara, M. A. Moreno, E. Pardo, \emph{Nonstable $K$-theory for graph
algebras}, Algebr. Represent. Theory {\bf 10} (2007), 157--178.

\bibitem{ADEL} S. Armstrong, K.  Dykema, R. Exel, H.  Li,
\emph{On embeddings of full amalgamated free product C*-algebras},
Proc. Amer. Math. Soc. {\bf 132} (2004), 2019--2030.

\bibitem{avitzour} D. Avitzour, \emph{Free products of C*-algebras},
Trans. Amer. Math. Soc. {\bf 271} (1982), 423--435.

\bibitem{BN} B Brenken, Z. Niu, \emph{The C*-algebra of a partial isometry}, Proc. Amer. Math. Soc.
{\bf 140} (2012), 199--206.

\bibitem{Brown} L. G. Brown, \emph{Ext of certain free product $C^{\ast}
$-algebras}, J. Operator Theory {\bf 6} (1981), 135--141.

\bibitem{BO} N. P. Brown, N. Ozawa, C*-algebras and
finite-dimensional approximations. Graduate Studies in Mathematics,
88. American Mathematical Society, Providence, RI, 2008.

\bibitem{ChoDyk} M. Choda, K. J. Dykema, \emph{Purely infinite simple
C*-algebras arising from free product constructions. III}, Proc.
Amer. Math. Soc. {\bf 128} (2000), 3269--3273.

\bibitem{Choi} M. D. Choi, \emph{A simple C*-algebra generated by
two finite-order unitaries}, Canad. J. Math.  {\bf 31} (1979),
867--880.

\bibitem{DykemaTAMS} K. J. Dykema, \emph{Simplicity and the stable
rank of some free product C*-algebras}, Trans. Amer. Math. Soc. {\bf
351} (1999), 1--40.

\bibitem{DykemaScan} K. J. Dykema, \emph{Purely infinite simple
C*-algebras arising from free product constructions, II}, Math.
Scand. {\bf 90} (2002), 73--86.

\bibitem{DyHaagRor} K . J. Dykema, U. Haagerup, M. R\o rdam,
\emph{The stable rank of some free product C*-algebras}, Duke Math.
J. {\bf 90} (1997), 95--121.

\bibitem{DykemaRordamCJM} K. J. Dykema, M. R\o rdam,  \emph{Purely infinite simple
C*-algebras arising from free product constructions}, Canad. J.
Math. {\bf 50} (1998), 323--341.

\bibitem{DykemaRordamGFA} K. J. Dykema, M. R\o rdam,
\emph{Projections in free product C*-algebras}, Geom. Funct. Anal.
{\bf 8} (1998), 1--16.

\bibitem{Germain} E. Germain, \emph{$KK$-theory of reduced free-product C*-algebras}
Duke Math. J. 82 (1996), 707--723.

\bibitem{ivanov} N. A. Ivanov, \emph{On the structure of some reduced amalgamated
free product C*-algebras}, Internat. J. Math. {\bf 22} (2011),
281--306.

\bibitem{Kishi} A. Kishimoto, \emph{Outer automorphisms and reduced crossed products of
simple C*-algebras}, Comm. Math. Phys. {\bf 81} (1981), 429--435.

\bibitem{lea}
W. G. Leavitt. {\it The module type of a ring.}
 Trans. Amer. Math. Soc. {\bf 103} (1962) 113--130.


\bibitem{McCla1} K. McClanahan, \emph{$C\sp *$-algebras generated by elements of a
unitary matrix}, J. Funct. Anal. {\bf 107} (1992), 439--457.

\bibitem{McCla2} K. McClanahan, \emph{$K$-theory and ${\rm
Ext}$-theory for rectangular unitary $C\sp *$-algebras}, Rocky
Mountain J. Math. {\bf 23} (1993), 1063--1080.

\bibitem{McCla3} K. McClanahan, \emph{Simplicity of reduced amalgamated products of
C*-algebras} Canad. J. Math. {\bf 46} (1994), 793--807.

\bibitem{Raeburn} I. Raeburn, Graph algebras. CBMS Regional Conference Series in
Mathematics, 103. Published for the Conference Board of the
Mathematical Sciences, Washington, DC; by the American Mathematical
Society, Providence, RI, 2005.

\bibitem{RV} M. R\o rdam, J. Villadsen, \emph{On the ordered
$K_0$-group of universal, free product C*-algebras}, K-Theory {\bf
15} (1998), 307--322.


\bibitem{voicu}
D. Voiculescu, \emph{Symmetries of some reduced free product
$C^\ast$-algebras}, Operator algebras and their connections with
topology and ergodic theory (Busteni, 1983), 556--588, Lecture Notes
in Math., 1132, Springer, Berlin, 1985.

\end{thebibliography}
\end{document}